\documentclass[12pt]{article}

\usepackage[dvipsnames]{xcolor}

\usepackage{amssymb,mathrsfs,amsmath}
\usepackage{amscd,amsmath,amsthm,amssymb,amsfonts,enumerate}
\usepackage{subfig,graphicx}

\usepackage{hyperref}
\hypersetup{
   colorlinks=true,
    linkcolor=red,
    filecolor=magenta,
   urlcolor=cyan,
    citecolor=blue,
}

\def\({\left(}
\def\]{\right]}
\def\[{\left[}
\def\){\right)}

\newtheorem{thm}{Theorem}[section]

\newtheorem{lem}[thm]{Lemma}
\newtheorem{rem}[thm]{Remark}

\newtheorem{exm}[thm]{Example}

\def\P{{\mathbb P}}
\def\E{{\mathbb E}}
\def\Cov{{\mathbb Cov}}
\def\SS{\mathcal S}

\newcommand{\disp}{\displaystyle}
\newcommand{\bea}{$$\begin{array}{ll}}
\newcommand{\eea}{\end{array}$$}
\newcommand{\bed}{\begin{displaymath}}
\newcommand{\eed}{\end{displaymath}}
\newcommand{\ad}{&\!\!\!\disp}
\newcommand{\aad}{&\disp}
\newcommand{\barray}{\begin{array}{ll}}
\newcommand{\earray}{\end{array}}
\newcommand{\beq}[1]{\begin{equation} \label{#1}}
\newcommand{\eeq}{\end{equation}}

\newcommand{\bedd}{\bed\begin{array}{l}}
\newcommand{\eedd}{\end{array}\eed}
\newcommand{\al}{\alpha}
\newcommand{\sg}{\sigma}

\newcommand{\e}{\varepsilon}

\newcommand{\nd}{\noindent}

\newcommand{\M}{{\cal{M}}}
\newcommand{\one}{{1}}

\newcommand{\wdt}{\widetilde}

\newcommand{\cd}{(\cdot)}
\newcommand{\rr}{\mathbb R}
\newcommand{\R}{\mathbb R}

\def\one{{\hbox{1{\kern -0.35em}1}}}

\def\ka{\kappa}

\def\L{{\cal L}}
\def\G{{\cal G}}

\def\({\left(}
\def\]{\right]}
\def\[{\left[}
\def\){\right)}
\def\one{{\hbox{1{\kern -0.35em}1}}}

\makeatletter \@addtoreset{equation}{section}

\def\para#1{\vskip 0.4\baselineskip\noindent{\bf #1}}
\def\qed{$\qquad \Box$}

\def\argmax{\hbox{argmax}}

\begin{document}

\title{Harvesting of a Stochastic Population under a Mixed Regular-Singular Control Formulation}

\author{Ky Q. Tran\thanks{Department
of Applied Mathematics and Statistics,
State University of New York - Korea Campus, Yeonsu-Gu, Incheon, Korea 21985, ky.tran@stonybrook.edu. The research of this author was supported by the National Research Foundation of Korea grant funded by the Korea Government (MIST) NRF-2021R1F1A1062361.}
\and
Bich T. N. Le \thanks{Department of Mathematics, Hue University of Education, Hue University, Hue city, Vietnam, ltnbich@hueuni.edu.vn. }
\and
George Yin\thanks{Department of Mathematics, University of Connecticut, Storrs, CT
06269, USA, gyin@uconn.edu.  The research of this author was supported in part by the Air Force Office of Scientific Research.}
}

\date{}

\maketitle

\begin{abstract}
This work focuses on optimal harvesting-renewing for a stochastic population. A mixed regular-singular control formulation with a state constraint
and regime-switching is introduced.
The decision makers either harvest or renew with finite or infinite harvesting/renewing rates.
The 
payoff
functions depend
on the harvesting/renewing rates. Several properties of the value functions are established.
The limiting value function as the white noise intensity approaches infinity is identified.
The
Markov chain approximation method is used to find numerical approximation of the value function and
optimal strategies.

\vskip 0.3 true in \nd{\bf Key words.}
Harvesting problem, controlled diffusion, singular control, state constraint, Markovian switching


\end{abstract}

\maketitle

\setlength{\baselineskip}{0.22in}

\newpage

\section{Introduction}

This work focuses on harvesting and renewing strategies for stochastic ecosystems  that are represented by controlled
stochastic differential equations.  Mathematically, the problem we consider
belongs to a
class of singular stochastic control problems, namely, harvesting-type problems.
Such problems have been studied extensively in various settings for different domain of applications; see \cite{A1998,A2000,Lande95,Hau1995,Hau19952,Jin12,K1984,K1985,L1997,RS1996}. To mention just a few of the recent developments, we refer to \cite{Zhu11,Ky17}
for single species and interacting population systems with regime switching.
 The paper \cite{H2019b} focuses on sustainable harvesting policies under long-run average criteria, which are further studied in \cite{L2020}; see also \cite{N2020} for a related work on a predator-prey system and \cite{Kunwai21}
 for an ergodic two-sided singular control formulation. In \cite{H2019,H2020}, the authors study ecosystems in which
 both renewing and harvesting actions are included. Optimal exploitation problems of renewable natural resources, which are harvesting-type problems, are studied in \cite{K2020,Ky21}. Related works on a general one-dimensional diffusion that is reflected at zero can be found in \cite{F2019} and references therein. Intensive treatment of stochastic population systems and hybrid diffusions can be found in \cite{Mao2006,YinZ10}, while applications in various areas of singular stochastic control
can be found in \cite{Fleming93,Pham09} and many references therein.

In this work, we propose a generalized harvesting model for  a stochastic population.
 The controller can perform either a regular control or a singular control to harvest and renew the species.
Moreover, we also consider the control objective
associated with renewing and harvesting. It should be noted that in optimal harvesting formulations in \cite{A1998,A2000,Zhu11,H2019,H2020}, the control cost is simply combined with the price functions.
As a result, it is frequently seen that when the manager decides to harvest (resp., renew), she should do that with the maximal possible harvesting rates (resp. renewing rates); see \cite{A1998,H2019,H2020}. In \cite{K2020,Ky21}, interesting phenomena appear when the control objective
is taken into account. For instance, the maximal possible rates are no longer optimal for harvesting and renewing in certain cases.
With the generalized formulation proposed in this work, the distinctions are even more pronounced.  In addition, as
a new twist,
state constraint is considered
in this work.
In particular, the time horizon of the control problem is $[0, \tau]$ where $\tau$ is the first time the population process is
below
a predetermined level. Another important
issue of interest is the impact of a white noise with large intensity. For Kolmogorov-type ecosystems, it is known that very large white noise make the species extinct, and have a major impact on harvesting actions; see \cite{A1998,Mao2006,Ky17}.
A novelty of the paper is the identification of
the limiting value function as the white noise intensity approaches infinity for a general stochastic population.

 In contrast to the existing results, our new contributions in this
 paper are as follows.
 (i) We formulate a harvesting problem with renewing and the consideration of control objective,
 a state constraint, and regime-switching. Both bounded and unbounded harvesting-renewing rates are allowed.
 (ii) We establish the finiteness and the continuity of the value function. We show that in common cases, it is optimal to keep the population size in a compact set. (iii) We study the impact of a white noise with large intensity on harvesting.
  (iv)  Based on the Markov chain approximation method, we construct a controlled
 Markov chain to approximate the given controlled population. It enables us to approximate the value function and near-optimal strategies.

 The rest of our work is organized as follows. Section 2 begins with
 the problem formulation. Section 3 focuses on properties of the value function and the impact of a white noise with large intensity.  In Section 4, we construct a controlled Markov chain to approximate the given controlled population system.
  Finally, the paper is concluded
 with several numerical examples for illustration together with
 additional remarks in the last section.
 
 \section{Formulation}\label{sec:for}
We work with a complete filtered probability space $(\Omega, \mathcal{F}, \P,
\{\mathcal{F}_t\})$ with the filtration $\{\mathcal{F}_t\}$ satisfying the usual condition
(i.e., it is right-continuous and ${\mathcal F}_0$ contains all the null sets). Let $\rr_+=[0, \infty)$ and $\mathbb{Z}_+=\{0, 1, 2, \dots\}$.
For a real number $x$, we denote $x^+=\max\{x, 0\}$ and $x^-=\max\{-x, 0\}$. Thus, $x=x^+-x^-$ and $|x|=x^++x^-$.
 Suppose that the population size $\xi(t)$ of a species at time $t$ is given by
\beq{e.1} d \xi(t)=b\big(\xi(t), \Lambda(t)\big)
dt+\sigma\big(\xi(t),  \Lambda(t)\big) d w(t).\eeq
 In the model, $w\cd$ is a
one-dimensional standard Brownian motion and $\Lambda\cd$ is a continuous time Markov chain. Moreover, $\Lambda\cd$ and $w\cd$ are independent and $\{\mathcal{F}_t\}$-adapted.
Suppose $\Lambda\cd$ takes values in $\mathcal{M}=\{1, 2, \dots, m_0\}$  with
generator $\Gamma=(\Gamma_{ \al \ell})_{m_0\times m_0}$ and $m_0$ being a positive integer.
The coefficients  $b(\cdot, \cdot)$ and $\sg(\cdot, \cdot)$ are real-valued functions defined on $\R_+\times \M$. 
The transition probabilities of
$\Lambda\cd$  is described by
\beq{e.main2}
\barray  \P\{\Lambda(t+\Delta t)=\ell |\Lambda(t)=\al \}
= \begin{cases} \Gamma_{\al \ell}\Delta t + o(\Delta t) &\ \ \text{if}
\ \ \al\not =\ell,  \\
1 + \Gamma_{\al \al}\Delta t + o(\Delta t) &\ \ \text{if} \ \ \al=\ell.
\end{cases}
\earray
\eeq
Note that $\Gamma_{\al \ell}\ge 0$ if $\al\ne \ell$ and $\sum_{\ell\in \M} \Gamma_{\al \ell}=0$ for any $\al\in \M$.
To illustrate, consider the stochastic logistic population growth model given by
\beq{e.main3}dX(t) = X(t) \Big(\ka_1(\Lambda(t))  - \ka_2(\Lambda(t)) X(t)\Big)dt + \ka_3(\Lambda(t)) X(t)dw(t),\eeq
where $\ka_1\cd, \ka_2\cd, \ka_3\cd$ are real-valued functions defined on the state space $\M$ of the Markov chain $\Lambda\cd$.
The switching component $\Lambda\cd$ is introduced to capture   the major environmental shifts (daily or seasonal changes or catastrophes) leading to the changes in the carrying capacities and interactions in different environments.
To visualize the dynamics of
\eqref{e.main3},
without loss of generality, assume that $\Lambda(0) = \al$. Then the
Markov chain rests in state $\al$ for an exponentially distributed random duration,
in this time interval,
the model  given by \eqref{e.main3} obeys
$$
dX(t) = X(t) \Big(\ka_1(\al)  - \ka_2(\al) X(t)\Big)dt + \ka_3(\al) X(t)dw(t),$$
until the Markov chain $\Lambda\cd$ jumps to another state $\ell$. Then the model \eqref{e.main3} obeys
$$
dX(t) = X(t) \Big(\ka_1(\ell)  - \ka_2(\ell) X(t)\Big)dt + \ka_3(\ell) X(t)dw(t)$$
for an exponentially distributed random
time until the Markov chain $\Lambda\cd$ jumps to a new state again
and so on.

 To proceed, we introduce the generator of the process $\big(\xi(t), \Lambda(t)\big)$.
For a function $\Phi(\cdot, \cdot): \rr\times \M\mapsto \rr$ satisfying
$\Phi(\cdot, \al)\in C^{2}(\rr)$ for each $\al\in \M$, we define
\bea
\L \Phi(x, \al)\ad= b( x, \al) \Phi'(x, \al)+\dfrac{1}{2}\sigma^2(x, \al)\Phi''(x, \al) + \sum\limits_{\ell\in \M}\Gamma_{\al \ell} \Phi(x, \ell),\eea
where $\Phi'(\cdot, \al)$ and $\Phi''(\cdot, \al)$ denote the first and second derivatives
(w.r.t. $x$) of $\Phi(\cdot, \al)$, respectively.

Next, we suppose that the population can be instantaneously harvested, 
instantaneously renewed, 
harvested with bounded rates, or renewed with bounded rates. In order to harvest or renew instantaneously (that is, harvest or renew with infinite rates), the controller needs to exercise an impulsive control. Meanwhile, to
harvest or renew with bounded rates, the controller performs a regular control.
Specifically,
we
assume
the dynamics of the species is given by
\beq{e.2}
\barray
X(t)\ad =x+\int_{0}^t b\big(X(s), \Lambda(s)\big) ds +  \int_{0}^t  \sigma\big( X(s), \Lambda(s)\big) dw(s) \\
\ad \qquad \quad - \int_{0}^t f\big(X(s), C(s)\big)ds - Y(t) + Z(t),
\earray
\eeq
where $x\in \rr_+$,
$X(t)$ is the population size  at time $t\ge 0$, $f:  \rr_+\times \mathcal{U}\mapsto \R$ is the harvesting-renewing rate corresponding to the control $C\cd$ taking values in a nonempty compact set $\mathcal{U}$ in $\mathbb{R}$, while $Y\cd$ and $Z\cd$ are impulsive controls.
In particular,
$Y(t)$ denotes the amount of the species  that has been instantaneously harvested up to time $t$, while $Z(t)$ denotes the amount of the species  that has been instantaneously renewed up to time $t$.

\para{Notation.} For each time $t$,
$X(t-)$ represents the state
before harvesting or renewing starts at time $t$, while $X(t)$ is the state immediately after. We assume the initial population size
to be
$X(0-)=x$ and initial regime
to be
$\Lambda(0)=\al$, respectively.
Hence $X(0)$ may not be equal to $X(0-)$ due to an impulsive harvesting $Y(0)$ or an impulsive renewing $Z(0)$. Throughout the paper we use the convention that $Y(0-)= Z(0-)=0$.
The jump sizes of $Y\cd$ and $Z\cd$ at time $t$ are denoted by $\Delta Y(t)=Y(t) -Y(t-)$ and $\Delta Z(t)=Z(t) -Z(t-)$, respectively. Thus,
$$Y(t) = \sum\limits_{0\le s\le t} \Delta Y(s), \ \  Z(t) = \sum\limits_{0\le s\le t} \Delta Z(s).$$
Suppose $\lambda\in \R_+$ and denote
$$\SS = \{x\in \R_+:  x\ge \lambda \}.$$
In this work,
we consider the harvesting problem on the time horizon $[0, \tau]$, where $$\tau=\inf\{t\ge 0: X(t)\notin \SS\}.$$
The price per unit of the species is a positive constant $a_1$.
The harvesting and renewing control is costly. Consider a function $g:
\mathbb{R}_+\times \M\times \mathcal{U}\mapsto \mathbb{R}_+$. The accumulated cost of  the regular control  $C\cd$ is
$$\int_0^{\tau} e^{-\delta s }g\big(X(s),\Lambda(s), C(s)\big)ds,$$
 where $\delta> 0$ is the discount factor.
 Note that we
 separate cost and income here. Later we will take into account 
 of both and set up the payoff.
 Regarding the regular control $C\cd$, the accumulated income of selling the harvested amount is $\int_{0}^{{\tau}} e^{-\delta s} a_1  f^+\big(X(s),C(s)\big)ds$ and the accumulated expense of the renewed amount is $\int_{0}^{{\tau}} e^{-\delta s} a_1 f^-\big(X(s),C(s)\big) ds$.
For notational simplicity, we define the price-cost function $p: \mathbb{R}_+\times \M\times \mathcal{U}\to \R$ given by
\beq{p.fun}
p(x, \al, c) = a_1 f(x,c) - g(x, \al, c)\ \  \text{for} \ \  (x, \al, c)\in \R_+\times \M\times \mathcal U.\eeq
Then the payoff functional associated with the regular control $C\cd$ is
$$ \E_{x, \al} \Big[ \int_{0}^{{\tau}} e^{-\delta s}  p\big(X(s),  \Lambda(s), C(s)\big)ds\Big],
$$
where
  $\E_{x, \al}$ denotes the expectation with
$X(0-)=x$ and $\Lambda(0)=\al$.
Let
$a_2$ and  $a_3$ be two positive constants. Suppose that the cost of instantaneous harvesting a unit of the species is $a_2$, while the cost of instantaneous renewing a unit of the species is $a_3$.
Then the accumulated cost of
exercising the impulsive control $Y\cd$ is $\int_{0}^{{\tau}} e^{-\delta s} a_2  dY(s)$ while the accumulated income of selling the harvested amount by $Y\cd$ is $\int_{0}^{{\tau}} e^{-\delta s} a_1  dY(s)$.
Meanwhile,
 the accumulated cost of exercising the impulsive control $Z\cd$ is $\int_{0}^{{\tau}} e^{-\delta s} a_3  dZ(s)$ and
  the accumulated expense of the renewed amount by $Z\cd$ is $\int_{0}^{{\tau}} e^{-\delta s} a_1  dZ(s)$.
 Hence
the payoff functional associated with the impulsive controls $Y\cd$ and $Z\cd$ is
$$\E_{x, \al} \Big[ \int_{0}^{\tau} e^{-\delta s} (a_1-a_2)  dY(s) - \int_{0}^{\tau} e^{-\delta s} (a_1+a_3)  dZ(s)\Big].$$
We define
$$q = a_1 - a_2, \ \  r = a_1+a_3.$$
Then, for a harvesting-renewing strategy $\Psi\equiv(C, Y, Z)$, we define the \textit{payoff functional} as
\beq{e.6}\barray
J(x, \al, \Psi)=  \E_{x, \al} \ad \Big[ \int_{0}^{{\tau}} e^{-\delta s}  p\big(X(s),  \Lambda(s), C(s)\big)ds\\ \aad \  +\int_{0}^{\tau} e^{-\delta s} q  dY(s) - \int_{0}^{\tau} e^{-\delta s} r  dZ(s)\Big].\earray
\eeq

\para{Control strategy.}
Let $\mathcal{A}_{x, \al}$ denote the collection of all admissible controls with initial condition $X(0-)=x$, $\Lambda(0)=\al$. A harvesting-renewing strategy $\Psi=(C, Y, Z)$ will be in $\mathcal{A}_{x, \al}$ if it satisfies
 the following
conditions.

\begin{itemize}
\item[{\rm (a)}]  The processes $C\cd$, $Y\cd$, and $Z\cd$ are adapted to
$\sg\{w(s), \Lambda(s): 0\le s\le t\}$; $C\cd$ takes values in $\mathcal{U}$, $Y\cd$ and  $Z\cd$ are impulsive controls with non-decreasing, nonnegative, piecewise constant, and right-continuous sample paths; $\Delta Y (s)\Delta Z(s)=0$ for any $s\in \R_+$; $\Delta Y(s)=\Delta Z(s)=0$ for any $s\ge \tau$;
\item[{\rm (b)}]
System \eqref{e.2} has a unique solution $X\cd$ with $X(t)\in \SS$ for any $t\in [0, \tau]$;
\item[{\rm (c)}] {$0\le  J(x, \al, \Psi)<\infty$, where $J\cd$ is the functional defined in \eqref{e.6}. }

\end{itemize}
The problem we are
interested in is to maximize the
payoff functional and find an optimal strategy $\Psi^*=(C^*, Y^*,  Z^*)\in \mathcal{A}_{ x, \al}$ such that
\beq{e.5}
J(x, \al, \Psi^*)=V(x, \al):= \sup\limits_{\Psi\in \mathcal{A}_{x, \al}}J( x, \al, \Psi).
\eeq
The function $V(\cdot)$ is called the \textit{value function}.

The standing assumptions are given below.
\begin{itemize}\item[{\rm (A)}]
\begin{itemize}
 \item[{\rm (a)}]  For any $n\in \mathbb{Z}_+$, there exists a positive constant $K_{n}$ such that
 for any $x, y\in \R_+$ with $|x|\le n, |y|\le n$ and any $\al\in \M$,
 $$|b(x, \al) - b(y, \al)| + | \sg(x, \al) - \sg(y, \al)| \le K_{n}|x-y|.$$
Moreover, for each initial condition $(x, \al)\in \R_+\times \M$,
the population system \eqref{e.1} has a unique global solution.

\item[{\rm (b)}]
  The control set $\mathcal{U}$ is a nonempty compact set of real numbers, $0\in \mathcal{U}$, $a_1,a_2, a_3$ are positive constants and $a_1>a_2$.
The function $g(\cdot, \cdot, \cdot)$ is continuous and bounded on $\R_+\times \M\times \mathcal{U}$. The function $f(\cdot, \cdot)$ is continuous and bounded on $\R_+\times \mathcal{U}$ and $f(\cdot, 0)=0$.
\end{itemize}
\end{itemize}

\begin{rem}{\rm
For simplicity, we require the systems of equations having a global solution. In fact, the existence and uniqueness of global solutions for a large class of Kolmogorov systems are guaranteed under suitable conditions; see the recent work \cite{NNY21} and the references therein, and
also \cite{Mao2006} for various sufficient conditions so that the population system \eqref{e.1} has a unique global solution.
Of course, here we are dealing with a controlled system so some modifications are needed. Nevertheless, in order to concentrate on our main task, we choose to simply assume this condition.

The mixed regular-singular control formulation together with the consideration of a state constraint, cost functions allow us to take into account various aspects of harvesting-type problems that have not been considered to date.
For instance, depending on  the costs, one can choose to apply either an impulsive harvesting/renewing or harvest and renew through the regular control. }
\end{rem}

\begin{rem}
{\rm The consideration of $\lambda$ is motivated by sustainability.
To the
best of our knowledge, the available literature focuses on the case $\lambda=0$ in which the optimal or near-optimal harvesting strategies might drive the population process to a very low level or extinction; see \cite{A1998,H2019,Zhu11}.
Because of the state constraint, one needs to
 treat carefully control actions when the population size is close to $\lambda$.

The function $f:  \rr_+\times \mathcal{U}\mapsto \R$ is the harvesting-renewing rate corresponding to the control $C\cd$.
In \cite{H2020}, the authors studied
the case $f(x, c)=c$ and $g(x, \al, c)=0$ for any $(x, \al, c) \in \R_+\times \M\times \mathcal{U}$.
Motivated by the observations in  \cite[Section 3]{A1998} and \cite{K2020}, one can also take $f(x, c) = \min\{cx, \ka\}$ and $g(x, \al, c)=0$ for any $(x, \al, c) \in \R_+\times \M\times \mathcal{U}$, where $\ka$ is a positive constant. By considering the general form $f:  \rr_+\times \mathcal{U}\mapsto \R$,  our formulation is much more general than the aforementioned cases.  
}
\end{rem}

\section{Properties of the Value Function}\label{sec:pro}

This section is devoted to several properties of the value function. We begin with a lemma that allows us to establish the finiteness and construct upper bounds of the value function.

\begin{lem}\label{lem:1} Let {\rm (A)} be satisfied. Suppose that there exists a  function $\Phi:  \SS\times \M \mapsto \mathbb{R}_+$ such that $\Phi(\cdot, \al)\in C^{2}({\SS})$ for each $\al\in \M$ and
 $\Phi(\cdot, \cdot)$ solves the following coupled system of quasi-variational inequalities
	\beq{e.3.1}
	\max \bigg\{ \mathcal{G}\Phi(x, \al), q -  \Phi' (x, \al),  \Phi'(x, \al)-r
\bigg\}\le 0 \ \  \text{for} \ \  (x, \al) \in \SS\times \M, 
\eeq
where
$$\mathcal{G}\Phi(x, \al) = (\mathcal{L} -\delta ) \Phi (x, \al)+\max\limits_{c \in \mathcal{U}}\Big[p(x, \al, c) -  \Phi' (x, \al)  f(x, c)\Big].$$
Recall that $\delta$ is the discount factor given in the payoff functional \eqref{e.6}.
	 Then we have
$$V(x, \al)\le \Phi(x, \al) \ \  \text{for} \ \  (x, \al) \in \SS\times \M.$$
\end{lem}

\begin{proof}
For a fixed  $(x, \al)\in \SS \times \M$ and $\Psi=(C, Y, Z)\in \mathcal{A}_{x, \al}$, let $X$ denote the corresponding harvested process.  Choose $N$ sufficiently large so that $|x|<N$. Define
$$
\tau_N=\inf\{t\ge 0:   X(t)\ge N\}, \ \  T_N = N\wedge \tau_N \wedge \tau.
$$
Then \beq{e.3.2}
\tau_N\to \infty\ \  \text{and}\ \  T_N \to \tau \ \   \text{almost surely as } \ \  N \to\infty,
\eeq
where $\tau=\inf\{t\ge 0: X(t)\notin \SS\}.$
 Then Dynkin's formula leads to
\beq{e.3.1.x}
\begin{array}{ll}
\E \ad \big[e^{-\delta T_N}\Phi\big(X(T_N), \Lambda(T_N)\big)\big]=\Phi(x, \al)+ \E\int_{0}^{T_N} e^{-\delta s}(\L-\delta)\Phi\(X(s), \Lambda(s)\)ds\\
\ad \qquad \qquad -\E \int_{0}^{T_N} e^{-\delta s}\Phi'\(X(s), \Lambda(s)\) f\big(X(s), C(s)\big)ds \\
\ad \qquad \qquad+ \E \sum\limits_{0\le s\le T_N}e^{-\delta s}\Big[\Phi\(X(s), \Lambda(s-)\)-\Phi\( X(s-),\Lambda(s-)\)\Big].
\end{array}
\eeq
It follows from \eqref{e.3.1} that
\beq{e.3.1.y}
\barray
\ad \E\int_{0}^{T_N} e^{-\delta s}(\L-\delta)\Phi\(X(s), \Lambda(s)\)ds-\E \int_{0}^{T_N} e^{-\delta s}\Phi'\(X(s), \Lambda(s)\) f\big(X(s), C(s)\big) ds\\
\ad \quad \le -\E\int_{0}^{T_N} e^{-\delta s} p\(X(s), C(s), \Lambda(s)\)ds.
\earray
\eeq
For each $s\in [0, T_N]$,
by the mean value theorem, there exists a point $\wdt{X}(s)$ between $X(s)$ and $X(s-)$ such that
\bea\Phi\(X(s), \Lambda(s-) \)-\Phi\big(X(s-), \Lambda(s-)\big)\ad =-\Delta Y(s) \Phi'\big(\wdt{X}(s), \Lambda(s-)\big)\\
\ad \quad  + \Delta Z(s)  \Phi'\big(\wdt{X}(s), \Lambda(s-)\big).\eea
By \eqref{e.3.1}, we have
\beq{e.3.1.zz}
\barray
\Phi\big(X(s), \Lambda(s-)\big)-\Phi\big(X(s-), \Lambda(s-)\big) \le -q \Delta Y(s)  +r  \Delta Z(s).
\earray
\eeq
It follows from
\eqref{e.3.1.x}, \eqref{e.3.1.y}, \eqref{e.3.1.zz}, and the nonnegativity of $\Phi(\cdot, \cdot)$ that
\bea
\Phi(x, \al) \ad \ge  \E\int_{0}^{T_N} e^{-\delta s} p\big(X(s), \Lambda(s), C(s)\big) ds\\
\ad \qquad + \E \int_{0}^{T_N} e^{-\delta s} q d Y(s)  - \E \int_{0}^{T_N} e^{-\delta s} r  d Z(s).
\eea
Letting $N\to\infty$, it follows from \eqref{e.3.2} and the bounded convergence theorem that
$\Phi(x, \al)\ge J(x, \al, \Psi).$
Taking supremum over all $\Psi\in \mathcal{A}_{ x, \al}$, we obtain $\Phi(x, \al)\ge V(x, \al)$. The conclusion follows.
\qed
\end{proof}

Using Lemma \ref{lem:1}, we proceed to present an easily verifiable
condition for the finiteness of the value function.

\begin{thm}
	\label{thm:1} Let {\rm (A)} be satisfied.   Moreover, suppose that there is a positive constant $K$ such that \begin{equation}\label{cond-bi}b(x, \al)\le \delta x+K \ \  \text{for} \ \  (x, \al)\in \SS\times \M.\end{equation}
	Then there exist a positive constant $M$ such that
	$$V(x, \al)\le q x+ M \ \  \text{for} \ \  (x, \al)\in \SS\times \M.$$
\end{thm}

\begin{proof}
Define
	$$\Phi(x, \al)=	 q x+ \dfrac{Kq + \ka_0}{\delta} \ \  \text{for} \ \  (x, \al)\in \SS\times \M,$$
	where $$\ka_0 = \sup\limits_{(y, \al, c)\in \R_+\times\M\times \mathcal U}\Big( p(y, \al, c) - qf(y, c) \Big).$$ By assumption (A)(b), $\ka_0$ is finite.
	 Since $q = a_1 - a_2< r = a_1+a_3$ and $\Phi'(x, \al)=q$, it is clear that
	\beq{x.1}
	q -  \Phi' (x, \al)= 0, \ \  \Phi'(x, \al)-r <0 \ \ \text{for} \ \  (x, \al) \in \SS\times \M.
	\eeq
		By \eqref{cond-bi}, we have
	\beq{x.3}
	\barray
	\G\Phi(x, \al)
	\ad = b(x, \al) q -\delta \Big(q x+ \frac{Kq+\ka_0}{\delta} \Big)  +\max\limits_{c \in \mathcal{U}}\Big[p(x, \al, c) -  q  f(x, c)\Big]\\
	\ad \le (\delta x  + K)q - (\delta q x +  Kq + \ka_0) +\ka_0 \\
	\ad =0 \ \  \text{for} \ \  (x, \al) \in \SS\times \M.
	\earray
	\eeq
	It follows from \eqref{x.1} and \eqref{x.3} that
	$$
\max \bigg\{ \mathcal{G}\Phi(x, \al), q -  \Phi' (x, \al),  \Phi'(x, \al)-r
\bigg\}\le 0 \ \  \text{for} \ \  (x, \al) \in \SS\times \M.
$$
That is, $\Phi(\cdot, \cdot)$ solves the system of inequalities \eqref{e.3.1}. By virtue of Lemma \ref{lem:1}, $V(x, \al)\le \Phi(x, \al)$ for any $(x, \al)\in \SS\times \M.$
 This completes the proof.
\qed
	\end{proof}
	
	Next, we establish the continuity of the value function.

\begin{thm} \label{thm:2} Let {\rm (A)} be satisfied.
	Then the following assertions hold.
	\begin{itemize}
	\item[\rm (a)] For any $x, y\in  \SS$ and $\al\in \M$,
	\beq{e.3.5}V(x, \al)\ge q  (x-y)^+ -r  (y-x)^+    + V(y, \al).\eeq
	
	\item[\rm (b)] $V(\cdot)$ is Lipschitz continuous on $\SS\times \M$.
	\end{itemize}
\end{thm}

\begin{proof}\
	
(a) Fix  $\Psi=(C, Y, Z)\in \mathcal{A}_{y, \al}$. Define
$$\widetilde C(t)=C(t), \ \  \widetilde Y(t)=Y(t) + (x-y)^+, \ \  \widetilde Z(t)=Z(t) + (y - x)^+, \ \   t\ge 0,$$
and $\wdt \Psi = (\wdt C, \wdt Y, \wdt Z)$.
Then $\wdt \Psi \in \mathcal{A}_{x, \al}$
and
$$J(x, \al, \widetilde{\Psi})=q(x-y)^+ -r (y-x)^+    + J(y, \al, \Psi).$$
 	Since $V(x, \al)\ge J(x, \al, \widetilde{\Psi})$, we have
 	$$V(x, \al)\ge q (x-y)^+ -r (y-x)^+    + J( y, \al, \Psi),$$
	from which, \eqref{e.3.5}  follows by taking supremum over $\Psi\in \mathcal{A}_{y, \al}$.
	
	(b) Similar to \eqref{e.3.5}, we have
	\beq{e.3.6}V(y, \al)\ge q(y-x)^+ -r  (x-y)^+    + V(x, \al).\eeq
	In view of \eqref{e.3.5}, \eqref{e.3.6}, for any $x, y\in \SS$ and $\al\in \M$,
		\bea|V(x, \al)-V(y, \al)|\ad \le \big(|q| + |r|\big) |x-y|.\eea
	Thus, $V(\cdot)$ is Lipschitz continuous on $\SS\times \M$.
	\qed
	\end{proof}

For population models given by diffusion processes, under optimal or near-optimal harvesting strategies, one should keep the population sizes in a bounded set. In other words, if the initial population is too high, an impulsive harvesting should be performed instantaneously; see \cite{A1998,H2019,H2020}. We proceed to provide a proof of that result for a general setting of a controlled regime-switching diffusions with a mixed regular-singular control.

\begin{thm}\label{thm:3}
Suppose that there exists a number $U>\lambda$ such that
\beq{he}
 q\Big( b(x, \al) - \delta (x-U)\Big)
+ \sup\limits_{c\in \mathcal U}\Big(a_2  f(x,c) - g(x,\al, c)\Big)<0, \ \  (x, \al)\in (U, \infty)\times \M.\eeq
Then for each $(x, \al)\in (U, \infty)\times \M$,
	$$V(x, \al)= V(U, \al) + q (x-U).$$
\end{thm}

\begin{rem} {\rm
By assumption (A)(b),
$$\ka_1:= \sup\limits_{(x, \al, c)\in \R_+\times \M\times \mathcal U}\Big(a_2  f(x,c) - g(x,\al, c)\Big)<\infty.$$
 If $\limsup\limits_{x\to \infty}b(x, \al)< - \ka_1/q$
 for each $\al \in \M$,
  then we can take $$U =\sup\big\{x>\lambda+1: \sup_{\al \in \M} b(x, \al)\ge  - \ka_1/q \big\}.$$
}
\end{rem}

\begin{proof}
	Fix some $(x, \al) \in  (U, \infty)\times \M$,
$\Psi=(C, Y, Z)\in \mathcal{A}_{x, \al}$, and let $X$
be the corresponding harvested process.
	Let $\e\in (0, 1)$ be a constant and
	define  \beq{5e:2}\Phi(y, \al)=q (y-U) +\e \ \  \text{for} \ \  (y, \al)\in [U, \infty)\times \M.\eeq
 We can extend $\Phi(\cdot, \cdot)$ to the entire $\SS\times \M$ so that $\Phi(\cdot, \al)$ is a $C^{2}$ function for each $\al\in \M$, and $\Phi(y, \al)> 0$  for all $(y, \al)\in  \SS \times \M$.
	 By assumption \eqref{he}, we can check that
	 \beq{}\notag
	 (\mathcal{L}-\delta)\Phi(y, \al) + a_2  f(y,c) - g(y, \al, c)< 0,  \ \  (y, \al, c)\times [U, \infty) \times \M\times \mathcal{U}.\eeq
	 Thus,
	 \beq{new.a}
\barray
 (\mathcal{L}-\delta)\Phi(y, \al) - qf(y, c)\ad < - \big[ a_2 f(y, c) - g(y, \al, c)\big] - qf(y, c)\\
 \ad = - \big[(a_2+q)f(y, c) - g(y, \al, c)\big]\\
 \ad = - p(y, \al, c), \ \  (y, \al, c)\times [U, \infty) \times \M\times \mathcal{U}.
\earray
\eeq
	 For an integer $N$ satisfying $N>U$,
	we define
	$$\tau_N=\inf\{t\ge 0:   X(t)\ge N\}, \ \ 	\wdt \gamma_U=\inf\{t\ge 0:    X(t)\le U\}, \ \  T_N = N\wedge \tau_N\wedge \wdt \gamma_U.
	$$
We have
$T_N \to \wdt \gamma_U$ almost surely as $N \to\infty$.
Note that  $\wdt \gamma_U\le \tau$. 	
By Dynkin's formula, 	\beq{new.b}\barray
	\E \ad \big[e^{-\delta T_N}\Phi\( X(T_N), \Lambda(T_N)\)\big]-\Phi(x, \al)= \E\int_{0}^{T_N} e^{-\delta s}(\L-\delta)\Phi\(X(s), \Lambda(s)\)ds\\
\ad \qquad  -\E \int_{0}^{T_N} e^{-\delta s}\Phi'\(X(s), \Lambda(s)\) f\big(X(s), C(s)\big) ds \\
\ad \qquad
+ \E \sum\limits_{0\le s\le T_N}e^{-\delta s}\Big[\Phi\(X(s), \Lambda(s-)\)-\Phi\(X(s-), \Lambda(s-)\)\Big].
	\earray
	\eeq	
For each $s\in [0, T_N]$, we have
\beq{e.3.1.z}
\barray
\Phi\(X(s), \Lambda(s-)\)-\Phi\(X(s-), \Lambda(s-)\)\ad = -q\Delta Y(s)  + q \Delta Z(s)\\
\ad \le -q \Delta Y(s)  + r  \Delta Z(s).
\earray
\eeq
	We obtain from \eqref{new.a}, \eqref{new.b}, and \eqref{e.3.1.z} that
	\beq{3e.3.3}\barray
	\E \ad \big[e^{-\delta T_N}\Phi\(X(T_N), \Lambda(T_N)\)\big]-\Phi(x, \al) \le
 -\E\int_{s_0}^{T_N} e^{-\delta s} p\big(X(s),  \Lambda(s), C(s)\big) ds \\
 \ad  \qquad - \E \sum\limits_{0\le s\le T_N}e^{-\delta s} q\Delta Y(s)  + \E \sum\limits_{0\le s\le T_N}e^{-\delta s}r \Delta Z(s).
	\earray
	\eeq
	Since $\Phi(y, \al)>0$  for any $(y, \al)\in  \SS\times \M$, it follows from \eqref{3e.3.3}  that
	\bea
	\ad \E\int_{0}^{T_N} e^{-\delta s}p\big(X(s),  \Lambda(s), C(s)\big)ds
\\
\ad \qquad \qquad + \E \sum\limits_{0\le s\le T_N}e^{-\delta s} q \Delta Y(s)  - \E \sum\limits_{0\le s\le T_N}e^{-\delta s}r \Delta Z(s)\\
\ad \qquad \le \Phi (x, \al).
	\eea
	Letting $N\to\infty$,  we obtain
	\beq{}\barray
	\ad \E\int_{0}^{ \wdt \gamma_U } e^{-\delta s} p\big(X(s),  \Lambda(s), C(s)\big)ds
\\
\ad \qquad\qquad + \E \sum\limits_{0\le s\le \wdt \gamma_U }e^{-\delta s} q \Delta Y(s)  - \E \sum\limits_{0\le s\le \wdt \gamma_U}e^{-\delta s}r \Delta Z(s)\\
\ad \qquad \le \Phi (x, \al).
	\earray
	\eeq
	As a result,
\bea
J(x, \al, \Psi)\ad  \le  V(U, \al) + \Phi(x, \al)\\
\ad = V(U, \al) + q (x-U) +\e.
\eea
Letting $\e\to 0$ yields
\beq{4e:2}
\barray
J(x, \al, \Psi)\ad \le V(U, \al) + q(x-U).
\earray
\eeq
Since \eqref{4e:2} holds for any $\Psi\in \mathcal{A}_{x, \al}$,
\beq{4e:2aaa}
V(x, \al) \le V(U, \al) + q (x-U).
\eeq
	On the other hand, it is obvious (by harvesting instantaneously $x-U$ at time $t=0$) that
	\beq{4e:3}
	V(x, \al)\ge V(U, \al) +q   (x-U).\eeq
	In view of \eqref{4e:2aaa} and \eqref{4e:3}, $$V( x, \al)=V(U, \al) + q (x-U).$$  The conclusion follows.
	\qed
\end{proof}

We proceed to discuss the impact of large white noise. By using Lemma \ref{lem:1}, we construct an upper bound for the value function. Then we show that the value function approaches the upper bound as the white noise intensity approaches infinity. The  following result is motivated by the presence of multiplicative noise.

\begin{thm} \label{thm:4}
	\label{thm:4} Suppose that $\lambda>0$.
Moreover,
	there exist positive constants $\beta\in (0,1)$, $K$, and $N$ such that
	\beq{e:20}xb(x, \al) \le K(1+x^2), \ \  	b(x, \al)q -\delta q(x-\lambda)\le K(x^{\beta}+1),\eeq
	and
	$$|\sg(x, \al)|\ge Nx \ \  \text{for} \ \  (x, \al)\in \SS\times \M.$$
	Then
	$$\lim\limits_{N\to \infty}V(x, \al)=  q(x-\lambda),$$	
	uniformly on $ [\lambda, M]\times \M$ for any positive constant $M>\lambda$.
\end{thm}

\begin{proof} Let $M>\lambda$.
	For a fixed $\e \in (0, r-q)$,
	let $K_\e>0$ be sufficiently large such that $$\dfrac{(M+1)^\beta}{K_\e}\le \e, \ \  \dfrac{\beta}{K_\e}\le r-q-\e.$$
	Define
	$$\Phi(x, \al)=q(x-\lambda) +  \dfrac{(x+1)^{\beta}}{K_\e}, \ \  (x, \al)\in \SS\times\M.$$
	Detailed computations lead to
	\bea
	\aad \Phi'(x, \al) = q +\dfrac{\beta}{K_\e (x+1)^{1-\beta}},\\
	\aad  \Phi''(x, \al) =  -\dfrac{\beta(1-\beta)}{K_\e (x+1)^{2-\beta}},  \ \  (x, \al)\in \SS\times \M.
	\eea
	It is clear that
	\beq{e:21}q-\Phi'(x, \al)\le 0, \ \  \Phi'(x, \al)-r\le 0 \ \  \text{for} \ \   (x, \al)\in \SS\times \M.\eeq
	By assumption (A)(b),
	$$\ka_2:=\sup\limits_{(x, \al,c) \in \R_+\times \M\times \mathcal{U}}\Big[p(x, \al, c) -  \Phi' (x, \al)  f(x, c)\Big]<\infty.$$
		Hence,
	\beq{e:22}
	\barray
	\G \Phi(x, \al)\ad \le  b(x, \al) q  + \dfrac{\beta}{K_\e (x+1)^{1-\beta}}b(x, \al)  \\
	\ad -\dfrac{\beta(1-\beta)N^2x^{2}}{2K_\e(x+1)^{2-\beta}} - \delta q(x-\lambda)- \dfrac{\delta (x+1)^{\beta}}{K_\e} + \ka_2.\\
	\earray
	\eeq
	For $x\ge \lambda>0$, by the first inequality in \eqref{e:20},
	\bea\dfrac{\beta}{K_\e (x+1)^{1-\beta}}b(x, \al)\ad \le \dfrac{\beta}{K_\e x(x+1)^{1-\beta}}xb(x, \al)\\
	\ad \le \dfrac{\beta K(1+x^2)}{K_\e x^{2-\beta}} \\
	\ad \le \dfrac{\beta K}{K_\e }\Big( x^\beta + \dfrac{1}{\lambda^{2-\beta}}\Big).
	\eea
	It follows from \eqref{e:20} and \eqref{e:22} that for $x\ge \lambda$,
\bea \G \Phi(x, \al)\ad \le K(x^\beta+1)+ \dfrac{\beta}{K_\e (x+1)^{1-\beta}}b(x, \al) -\dfrac{\beta(1-\beta)N^2x^{2}}{K_\e(x+1)^{2-\beta}} + \ka_2\\
\ad \le \Big[K+\dfrac{\beta K}{K_\e}-\dfrac{\beta(1-\beta)N^2x^{2-\beta}}{2K_\e (x+1)^{2-\beta}}\Big]x^{\beta}+K+\dfrac{\beta K}{K_\e \lambda^{2-\beta}  } +\ka_2.
\eea
Hence there exists a positive number $N_0$ such that for any $N\ge N_0$,
\beq{e:24}
\G\Phi(x, \al)\le 0 \ \  \text{for} \ \  (x, \al)\in \SS\times \M.
\eeq
By \eqref{e:21}, \eqref{e:24}, and  Lemma \ref{lem:1}, we obtain
	$$V(x, \al)\le q(x-\lambda) +  \dfrac{(x+1)^{\beta}}{K_\e}, \ \  (x, \al)\in \SS\times \M$$
	provided that $N\ge N_0$.
	In particular, since $\dfrac{(M+1)^\beta}{K_\e}\le \e$, then for $N\ge N_0$, we have
	\beq{e:24xx}V(x, \al)\le q(x-\lambda) +\e \ \  \text{for}\ \  (x, \al)\in [\lambda, M]\times \M.\eeq
	 Let $\Psi=(C, Y, Z)\in \mathcal{A}_{x, \al}$ given by
	 $$C(t)=0, \ \  Y(t)=x-\lambda, \ \  Z(t)=0 \ \  \text{for} \ \  t\ge 0.$$
	 Then $J(x, \al, \Psi)=q(x-\lambda)$. It follows that
 $V(x, \al)\ge q(x-\lambda)$. This together with \eqref{e:24xx}
 yields
 that
 $$0\le V(x, \al)-q(x-\lambda)\le \e \ \  \text{for} \ \  (x, \al)\in [\lambda, M]\times \M,$$
 which leads to the desired conclusion.
 \qed
\end{proof}

We have a similar result regarding to the case of additive noise.

\begin{thm} \label{thm:5}
	\label{thm:5} Suppose that $\lambda>0$.
	 Moreover,
	there exist positive constants $K$ and $N$ such that
	\beq{e:25}	xb(x, \al) \le K(1+x^2),   \ \  b(x, \al)q -\delta q(x-\lambda)\le K,\eeq
	and
	$$|\sg(x, \al)|\ge N \ \  \text{for} \ \  (x, \al)\in \SS\times \M.$$
	Then
	$$\lim\limits_{N\to \infty}V(x, \al) = q(x-\lambda),$$	
	uniformly on $[\lambda, M]\times \M$ for any positive constant $M>\lambda$.
\end{thm}

\begin{proof} Let $M>\lambda$.
	For a fixed $\e\in (0, r-q)$,
	let $\beta\in (0,1)$ be such that
$\beta K<\delta$
	and
	let $K_\e>0$ be sufficiently large such that
	 $$\dfrac{\beta}{K_\e}\le r-q-\e, \ \  \dfrac{ (M+1)^\beta }{ K_\e }\le \e.$$
	Define
	$$\Phi(x, \al)=q(x-\lambda) +  \dfrac{(x+1)^{\beta}}{K_\e} \ \  \text{for}\ \  (x, \al)\in \SS\times\M.$$ 	
	We have
	\bea
	\aad \Phi'(x, \al) = q +\dfrac{\beta}{K_\e (x+1)^{1-\beta}},\\
	\aad  \Phi''(x, \al) =  -\dfrac{\beta(1-\beta)}{K_\e (x+1)^{2-\beta}}, \ \  (x, \al)\in \SS\times \M.
	\eea
	We can check that
	\beq{e:212}q-\Phi'(x, \al)\le 0, \ \  \Phi'(x, \al)-r\le 0 \ \  \text{for} \ \   (x, \al)\in \SS\times \M.\eeq
	By assumption (A)(b),
	$$\ka_3:=\sup\limits_{(x, \al,c) \in \R_+\times \M\times \mathcal{U}}\Big[p(x, \al, c) -  \Phi' (x, \al)  f(x, c)\Big]<\infty.$$	Hence,
	\beq{eee:40}
	\barray
	\G \Phi(x, \al)\ad \le  b(x, \al) q  + \dfrac{\beta}{K_\e (x+1)^{1-\beta}}b(x, \al)  \\
	\ad -\dfrac{\beta(1-\beta)N^2}{2K_\e(x+1)^{2-\beta}} - \delta q(x-\lambda)- \dfrac{\delta (x+1)^{\beta}}{K_\e} + \ka_3.
	\earray	\eeq
	For $x\ge \lambda>0$, by the first inequality in \eqref{e:25},
	\bea\dfrac{\beta}{K_\e (x+1)^{1-\beta}}b(x, \al)\ad \le \dfrac{\beta}{K_\e x(x+1)^{1-\beta}}xb(x, \al)\\
	\ad \le \dfrac{\beta K(1+x^2)}{K_\e x^{2-\beta}} \\
	\ad \le \dfrac{\beta K}{K_\e }\Big( x^\beta + \dfrac{1}{\lambda^{2-\beta}}\Big).
	\eea
	It follows from the second inequality in \eqref{e:25} and \eqref{eee:40} that for $x\ge \lambda$, 	\bea \G \Phi(x, \al)\ad \le K + \dfrac{\beta K}{K_\e }\Big( x^\beta + \dfrac{1}{\lambda^{2-\beta}}\Big)  -\dfrac{\beta(1-\beta)N^2}{2K_\e(x+1)^{2-\beta}} - \dfrac{\delta x^{\beta}}{K_\e}+ \ka_3\\
	\ad \le K +\dfrac{\beta K}{K_\e\lambda^{2-\beta}} + \ka_3 -
	\dfrac{\delta -\beta K}{K_\e}x^\beta
	-\dfrac{\beta(1-\beta)N^2}{2K_\e (x+1)^{2-\beta}}.
	\eea
	Since $\beta K<\delta$, we can choose a constant
	 $\lambda_1>\lambda$  such that
	$$K +\dfrac{\beta K}{K_\e\lambda^{2-\beta}} + \ka_3 -
	\dfrac{\delta -\beta K}{K_\e}x^\beta\le 0 \ \  \text{for} \ \  (x, \al)\in [\lambda
	_1, \infty)\times \M.$$
	There exists a positive number $N_0$ such that for any $N\ge N_0$,
	$$K +\dfrac{\beta K}{K_\e\lambda^{2-\beta}} + \ka_3 	-\dfrac{\beta(1-\beta)N^2}{2K_\e (x+1)^{2-\beta}}\le 0 \ \  \text{for}\ \  (x, \al)\in [\lambda, \lambda_1].$$
	Hence for $N\ge N_0$,
\beq{e:24x}
\G\Phi(x, \al)\le 0 \ \  \text{for} \ \   (x, \al)\in \SS\times \M.
\eeq
By	 \eqref{e:212}, \eqref{e:24x}, and Lemma \ref{lem:1}, we have 	$$V(x, \al)\le q(x-\lambda) +  \dfrac{(x+1)^{\beta}}{K_\e} \ \  \text{for} \ \  (x, \al)\in \SS\times \M$$
	provided that $N\ge N_0$.
	In particular, since $\dfrac{(M+1)^\beta}{K_\e}\le \e$,
for $N\ge N_0$,
	\beq{e:25a}V(x, \al)\le q(x-\lambda) +\e \ \  \text{for} \ \ (x, \al)\in [\lambda, M]\times \M.\eeq
	 Let $\Psi=(C, Y, Z)\in \mathcal{A}_{x, \al}$ given by
	 $$C(t)=0,  \ \  Y(t)=x-\lambda, \ \  Z(t)=0 \ \  \text{for} \ \ t\ge 0.$$
	 Then $J(x, \al, \Psi)=q(x-\lambda)$. It follows that
 $V(x, \al)\ge q(x-\lambda)$. This together with \eqref{e:25a} gives us that
 $$0\le V(x, \al)-q(x-\lambda)\le \e \ \  \text{for} \ \ (x, \al)\in [\lambda, M]\times \M,$$
 which leads to the desired conclusion.
 \qed
\end{proof}

\section{Numerical Approximation}
\label{sec:num}

Formally, the associated Hamilton-Jacobi-Bellman equation of the underlying problem is given by
 \beq{hjb}
 \barray
\aad \max \bigg\{ \mathcal{G}\Phi(x, \al), q - \Phi'(x, \al),\Phi'(x, \al) -r
\bigg\}=0 \ \  \text{for} \ \ (x, \al)\in \SS\times \M,\\
\aad \Phi(x, \al) =0 \ \  \text{for} \ \ (x, \al)\notin \SS\times \M,
\earray
\eeq
where
$$\mathcal{G}\Phi(x, \al) = (\mathcal{L} -\delta ) \Phi (x, \al)+\max\limits_{c \in \mathcal{U}}\Big[p(x, \al, c) - \Phi' (x, \al)  f(x, c)\Big].$$
A closed-form solution to \eqref{hjb} is virtually impossible to obtain. Thus, we
proceed with
 the Markov chain approximation method; see \cite{Jin12,KM91,Kushner92}. That is,
we
construct a discrete time, finite-state,
controlled  Markov chain
to approximate the controlled switching diffusions.
In \cite{Jin12}, the authors applied that method to solve a harvesting-type problem for a regime-switching diffusion with a regular control and a singular control. In this paper, we need to adopt and modify the approximation to fit the combination of a two-sided singular control and regular control formulation.
In view of Theorem \ref{thm:3},
 we only need to choose a large positive integer $U$ and compute the value function on $\SS\cap [0, U]=[\lambda, U]$. With $x\in [\lambda, U]$, we can rewrite \eqref{e.2} as
\beq{e.2.x}
\barray
 X(t) \ad =x+\int_{0}^t \big( b(X(s), \Lambda(s)
\big) -f\big(X(s),C(s) \big)ds +  \int_{0}^t  \sigma\big(X(s), \Lambda(s)\big) dw(s)\\
\ad \qquad\qquad  - Y(t) + Z(t).
\earray
\eeq
The payoff functional is
\beq{e.4.x}
\barray
\aad J( x, \al, \Psi)= \E \bigg[\int_{0}^{\tau} e^{-\delta s} p(X(s), \Lambda(s), C(s))ds  \\
\ad \qquad \qquad+ \int_{0}^{\tau} e^{-\delta s} qdY(s)-\int_{0}^{\tau} e^{-\delta s} r  dZ(s)\bigg].
\earray
\eeq

\subsection{Approximating Markov Chains}
Let $h$ be a discretization parameter for $X\cd$.
Assume without loss of generality that $U$ is a multiple of $h$.
Define
$$S_{h}: = [0, U]\cap \{x\in [\lambda - h, U]: x= k h, k\in \mathbb{Z}_+\}.$$
Let $\{(X^h_n, \Lambda^h_n): n\in \mathbb{Z}_+\}$
be a discrete-time controlled Markov chain with state space $S_{h}\times \M$.
For each $n$, the increments of the chain $\Delta X^h_n= X^h_{n+1}-X^h_n$ approximates exactly one of the following quantities dynamically.

\begin{itemize}
\item Diffusion step: $\big(b(X(t), \Lambda(t)) - f(X(t), C(t)\big)dt + \sg(X(t), \Lambda(t))dw(t)$.
\item Impulsive harvesting step: $-dY(t)$.
\item Impulsive renewing step: $dZ(t)$.
\end{itemize}
We use $\pi^h_n$ to denote the type of step $n$. We set $\pi^h_n=0$ if step $n$ is a diffusion step, $\pi^h_n=1$ if step $n$ is an impulsive harvesting step, and $\pi^h_n=-1$ if step $n$ is an impulsive renewing  step. Each of these steps is described precisely in what follows.

(a) If $\pi^h_n=0$, a regular control $C^h_n$ is exercised so that $\Delta X^h_n$ is to behave like an increment of $\int (b -f) dt +\sg dw$ over a small time interval.
(b) If $\pi^h_n=1$, a regular control $C^h_n$ and an impulsive harvesting $\Delta Y^h_n=h$ are applied.
(c) If $\pi^h_n=-1$, a regular control $C^h_n$ and an impulsive renewing $\Delta Z^h_n=h$ are applied.

Thus, the amounts of impulsive harvesting and impulsive renewing at any step $n$ are $\Delta Y^h_n = h I_{\{\pi^h_n=1\}}$
and $\Delta Z^h_n = h I_{\{\pi^h_n=-1\}}$, respectively.
Let $C^h =\{C^h_n\}$ be a sequence of regular controls. The space of regular controls is given by $\mathcal{U}$. Let $\pi^h = \{\pi_n^h\}_{n\ge 0}$ denote the sequence of control types.
We denote by $p^h\((x,\al),(y, \ell) |\pi, c\)$ the transition probability from state $(x, \al)$ to another state $(y, \ell)$ under the control type $\pi$ and regular control $c$.
Denote
$\mathcal{F}^h_n=\sigma\{X^h_k, \Lambda^h_k, \pi^h_k, C^h_k, k\le n\}$.

The sequence $\{(\pi^h_n, C^h_n)\}$
is said to be admissible if it satisfies the following conditions:
\begin{itemize}
	\item[{\rm (a)}]
	$\pi^h_n, C^h_n$ are
	$\sigma\{X^h_0, \dots, X^h_{n}, \Lambda^h_0, \dots, \Lambda^h_{n}, \pi^h_0, \dots, \pi^h_{n-1}, C^h_1, \dots, C^h_{n-1}\}-\text{adapted},$
	\item[{\rm (b)}]  For any $(x, \al)\in S_h\times \M$, we have
	\begin{equation*}
	\barray
	\P\{ X^h_{n+1} =x, \Lambda^h_{n+1}=\al | \mathcal{F}^h_n\}&= \P\{X^h_{n+1} = x, \Lambda^h_{n+1}=\al | X^h_n, \Lambda^h_n, \pi^h_n, C^h_n\} \\
	&= p^h\big( \big( X^h_n,\Lambda^h_n) ,  (x, \al) | \pi^h_n, C^h_n\big),
	\earray
	\end{equation*}
	\item[{\rm (c)}] $\pi^h_n \in \{0, \pm 1\}, C^h_n \in \mathcal{U}$, $X^h_n\in S_h, \Lambda^h_n\in \M$ for all $n\in \mathbb{Z}_+$.
\end{itemize}
The class of all admissible control sequences $(\pi^h, C^h)$ for initial state $(x, \al)$ will be denoted by
$\mathcal{A}^h_{x, \al}$.

For each
 $(x, \al,  i, c)\in S_h\times \M \times \{0, \pm1\}\times \mathcal{U}$,
we define
a family of the interpolation intervals $\Delta t^h (x,\al,  i, c)$. The values of $\Delta t^h (x, \al,  i, c)$ will be specified later. Then we define
\beq{e.4.3}
\barray
\aad t^h_0 = 0,\ \   \Delta t^h_k = \Delta t^h(X^h_k, \Lambda^h_k, \pi^h_k, C^h_k),
\ \   t^h_n = \sum\limits_{k=0}^{n-1} \Delta t^h_k.\\
\earray
\eeq
For $(x, \al)\in S_h\times \M$ and $(\pi^h, C^h)\in \mathcal{A}^h_{x, \al}$, the payoff functional for the controlled Markov chain is defined as
\beq{e.4.4}
\barray
J^h(x, \al, \pi^h, C^h) & =  \E\sum\limits_{k=0}^{\eta_h-1} e^{-\delta t_k^h}\bigg\{p(X^h_k, \Lambda^h_k, C^h_k) \Delta t^h_k + q  \Delta Y^{h}_k  -r \Delta Z^{h}_k \bigg\},
\earray
\eeq
with $\eta_h = \inf\{n\ge 0: X^h_n\notin \SS\}$.
The value function of the controlled Markov chain is
\beq{e.4.5}
V^h(x, \al) = \sup\limits_{(\pi^h, C^h)\in \mathcal{A}^h_{x, \al}} J^h ( x, \al, \pi^h, C^h).
\eeq
The corresponding dynamic programming equation for the discrete approximation is given by
$V^h(x, \al)=\max\{  V^{h}_{i, c}(x, \al): i\in \{-1, 0, 1\}, c\in \mathcal{U}\}$
for any $(x, \al)\in S_h\times \M$, where 
\begin{equation*}
\barray
\aad V^{h}_{1, c}(x, \al) = \big(  V^{h}(x-h, \al) + qh \big) I_{\{x>\lambda\}},\\
\aad V^{h}_{-1, c}(x, \al) =  \big( V^{h}(x+h, \al) - rh \big) I_{ \{x<U\}},\\
\aad V^{h}_{0, c}(x, \al) = e^{-\delta \Delta t^h(x, \al, 0, c)} \sum\limits_{(y, \ell)\in S_h} V^h(y, \ell)p^h\big( (x, \al), (y, \ell)| 0, c \big)I_{\{ x<U \} }\\
\aad \hspace{5cm}  + p(x, \al, c)\Delta t^h(x, \al, 0, c)I_{\{ x<U \} }.
\earray
\end{equation*}
Note that $V^h(x, \al) = 0$ for $(x, \al)\notin S_h\times \M$.

\subsection{Transition Probabilities}
Let $\E^{h, \pi, c}_{x,  \al, n}$, $\Cov^{h, \pi, c}_{ x, \al, n}$ denote the conditional expectation and covariance given by
$$\{X_k^h, \Lambda_k^h, \pi_k^h, C^h_k, k\le n-1, X_n^h=x,  \Lambda_n^h=\al, \pi^h_n=\pi, C^h_n=c\},$$
respectively. We proceed to define transition probabilities $p^h \big( (x, \al), (y, \ell) | \pi, c\big)$ so that the controlled Markov chain $\{(X^h_n, \Lambda^h_n)\}$ is locally consistent with respect to $(X\cd, \Lambda\cd)$
in the sense that the following hold.
\beq{e.4.5a}
\barray
\aad \E^{h, 0, c}_{x, \al, n}\Delta X_n^h = {b}(x, \al) \Delta t^h(x, \al, 0, c) + o(\Delta t^h(x, \al, 0, c)),\\
\aad Cov^{h, 0, c}_{x, \al, n}\Delta X_n^h = \sg^2(x, \al)\Delta t^h(x, \al, 0, c) + o(\Delta t^h(x, \al, 0, c)),\\
\aad \sup\limits_{n, \ \omega} |\Delta X_n^h| \to 0 \ \hbox{ as } \ h \to 0.
\earray
\eeq
To this end, motivated by the procedure in \cite{Jin12,Kushner92},
 we construct the transition probabilities below. For $(x, \al)\in S_h\times \M$, $\lambda \le x<U$ and $c\in \mathcal{U}$, we define
 $$Q_h(x, \al, 0, c) =  \sg^2(x, \al) +h |b(x, \al) - f(x, c)| -h^2\Gamma_{\al\al}+\zeta(h).$$
 We set $\zeta(h) = h$. If, in addition, $\sg^2(x, \al)>0$ for any $(x, \al)\in \SS\times \M$, we can simply take $\zeta(h)=0$. Then we define
 \beq{e.4.7}
\barray
\aad p^h \((x, \al) , (x+ h, \al) | 0, c\) =
\dfrac{ \sg^2(x, \al)/2 +\big(b(x, \al)-f(x ,c)\big)^+ h}{Q_h(x, \al, 0, c)}, \\
\aad p^h \((x, \al), (x-h , \al) | 0, c\) =
\dfrac{ \sg^2(x, \al)/2+\big(b(x, \al)-f(x ,c)\big)^- h }{Q_h(x, \al, 0, c)}, \\
\aad   p^h \( (x, \al), (x, \ell) | 0, c\) =\dfrac{h^2 \Gamma_{\al \ell} }{ Q_h (x, \al, 0, c)} \text{ for }\al\ne \ell , \\
\aad p^h \( (x, \al), (x, \al) | 0, c\) =\dfrac{\zeta(h)  }{ Q_h (x, \al, 0, c)},\ \  \Delta t^h( x, \al, 0, c)=\dfrac{h^2}{Q_h(x, \al, 0, c)},\\
\aad p^h \( (x,\al),  (y, \ell) |0, c\)=0 \ \  \text{for all not listed values of $(y, \ell) \in S_h\times \M$}.
\earray
\eeq
 At impulsive harvesting and renewing steps, we define
\beq{e.4.8}
\barray
& p^h\(  (x, \al) , (x - h, \al) | 1, c\)=1, \ \  \Delta t^h ( x, \al, 1, c) =0,\\
& p^h\( (x, \al), (x + h, \al) | -1, c\)=1, \ \  \Delta t^h ( x, \al, -1, c) =0.
\earray
\eeq
Thus, $p^h \( (x,\al),  (y, \ell) |\pm 1, c\)=0$ for all nonlisted values of $(y, \ell) \in S_h\times \M$.
The definition of
$p^h((x, \al), \cdot | i, c\big)$ for $(x, \al)\notin \SS\times \M$ is not important since in the analysis of the control problem, the chain will be stopped at the first time $\{X^h_n\}$ exits $\SS$.
Using the above transition probabilities,
we can check that
 the local
 consistency
 conditions of $\{(X^h_n, \Lambda^h_n)\}$ in \eqref{e.4.5a} are satisfied.

The convergence result is based on a continuous-time interpolation of the chain and relaxed control representations. In addition, a ``stretched-out'' time scale is introduced to overcome the possible non-tightness of the piecewise constant  interpolations associated with $Y\cd$ and $Z\cd$; see \cite{Jin12} and \cite{KM91}. Using weak convergence methods, we can obtain the convergence of the value function. The main convergence result is given below. The proof is a modification of that in \cite{Jin12}, we omit it for brevity.

\begin{thm}\label{thm:6}  Let $V(x, \al)$ and $V^h(x, \al)$ be value functions defined in \eqref{e.5} and \eqref{e.4.5}, respectively. Then $V^h(x, \al)\to V(x, \al)$ as $h\to 0$.
\end{thm}

\section{Numerical Examples}\label{sec:exm}

We consider a stochastic population with harvesting and renewing given by
\beq{e.exm} dX(t)=b(X(t), \Lambda(t)) dt +  \sigma( X(t), \Lambda(t)) dw(t) - C(t)dt - dY(t) + dZ(t),\eeq
where
\begin{equation*}\barray
& b(x, \al) = x(\al-1.5x), \ \  \sg(x, \al) = (\al/2)x.
\earray
\end{equation*}
Suppose that $\M=\{1, 2\}$ and  the generator $\Gamma=(\Gamma_{\al \ell})_{2\times 2}$ of the Markov chain $\Lambda\cd$ is given by $$\Gamma_{11}=-1, \ \  \Gamma_{12}=1, \ \  \Gamma_{21}=1.5, \ \  \Gamma_{22}=-1.5.$$
We also assume that
\begin{equation*}\barray
& q = 1, \ \   r = 2.25, \ \  \delta =0.05,\\
&  p(x, \al, c) = \dfrac{3c}{2}-\dfrac{\al c^2}{8(1+x)}   \ \  \text{for} \ \ (x, \al, c)\in \R_+\times \M\times \mathcal{U}, \\
& f(x, c) = c \ \  \text{for} \ \ (x,c)\in  \R_+\times \mathcal{U}.\\
\earray
\end{equation*}
The original price and costs are
$a_1=1.5, a_2=0.5, a_3 =0.75$ and the regular control cost function is $g(x, \al, c) =\frac{\al c^2}{8(1+x)}$.
Thus, for each $x\in \R_+$ and $\al\in \M$, $g(x, \al, c)$ has a quadratic form.
For a fixed regular control $c$ and regime $\al\in \M$, the cost function is decreasing with respect to the population size. This is motivated by an observation in harvesting problems that when the species is rare, it is more difficult
to harvest 
leading
to the higher harvesting cost; see \cite{A2000}.
By the state constraint, the time horizon of the control problem is $[0, \tau]$, where $\tau=\inf\{t\ge 0: X(t)<\lambda\}$. The value of $\lambda$ and the control set $\mathcal{U}$ are to be determined.
Using the transition probabilities constructed in \eqref{e.4.7} and \eqref{e.4.8}, we carry out the computation by using value iteration and policy iteration method.  For each $(x, \al)\in S_h\times \M$, we use the following notations for the $n$th iteration: $C^{h, n}(x, \al)$ is the regular control, $\pi^{h, n}(x, \al)$ is the control type, and $V^{h, n}(x, \al)$ is the value function. Initially, we take $$C^{h, 0}(x, \al)= 0, \ \   \pi^{h, 0}(x, \al)=1,  \ \  V^{h, 0}(x, \al) =q(x-\lambda) \ \  \text{for} \ \ (x,\al)\in S_h\times \M.$$
Note that $\big(\pi^{h, 0}(x, \al), C^{h, 0}(x, \al)\big)$
corresponds to the control $\Psi=(C, Y, Z)\in \mathcal{A}_{x, \al}$ given by
	 $$C(t)=0, \ \  Y(t)=x-\lambda, \ \  Z(t)=0 \ \  \text{for} \ \ t\ge 0,$$
of the controlled switching diffusion \eqref{e.2}.
Based on our computation of the $n$th iteration, we define
\begin{equation*}
\barray
\aad V^{h, n+1}_{1, c}(x, \al) = \big(  V^{h, n}(x-h, \al) + qh \big) I_{\{x>\lambda\}},\\
\aad V^{h, n+1}_{-1, c}(x, \al) =  \big( V^{h, n}(x+h, \al) - rh \big) I_{ \{x<U\}},\\
\aad V^{h, n+1}_{0, c}(x, \al) = e^{-\delta \Delta t^h(x, 0, c, \al)}
\big[ V^{h, n}(x +h, \al)p^h\( (x, \al) , (x+h, \al) | 0, c\)
\\
\aad \qquad+ V^{h, n}(x -h, \al)p^h\( (x, \al) , (x-h, \al) | 0, c\)
\\
\aad \qquad + V^{h, n}(x, \al)p^h\( (x, \al) , (x, \al) | 0, c\)
\big] I_{ \{x<U\}} + p(x, \al, c) \Delta t^h(x, \al, i, c) I_{ \{x<U\}}.
\earray
\end{equation*}
We find an improved value $V^{h, n+1}(x, \al)$
and record the corresponding controls by
\begin{equation*}
\barray
\aad \(\pi^{h, n+1}(x, \al), C^{h, n+1}(x, \al)\)=\argmax \left\{V^{h, n+1}_{i, c}(x, \al): i\in \{-1, 0, 1\}, c\in \mathcal{U}\right\} ,\\
\aad V^{h, n+1}(x, \al)= \max \left\{V^{h, n+1}_{i, c}(x, \al): i\in \{-1, 0, 1\}, c\in \mathcal{U}\right\}.
\earray
\end{equation*}
The iterations stop as soon as the
increment
$V^{h, n+1}(x, \al)-V^{h, n}(x, \al)$
reaches a pre-specified
tolerance level. We set the error tolerance to be $10^{-6}$.

 \begin{figure}[h!tb]
 	\centering		{{\includegraphics[width=5cm]{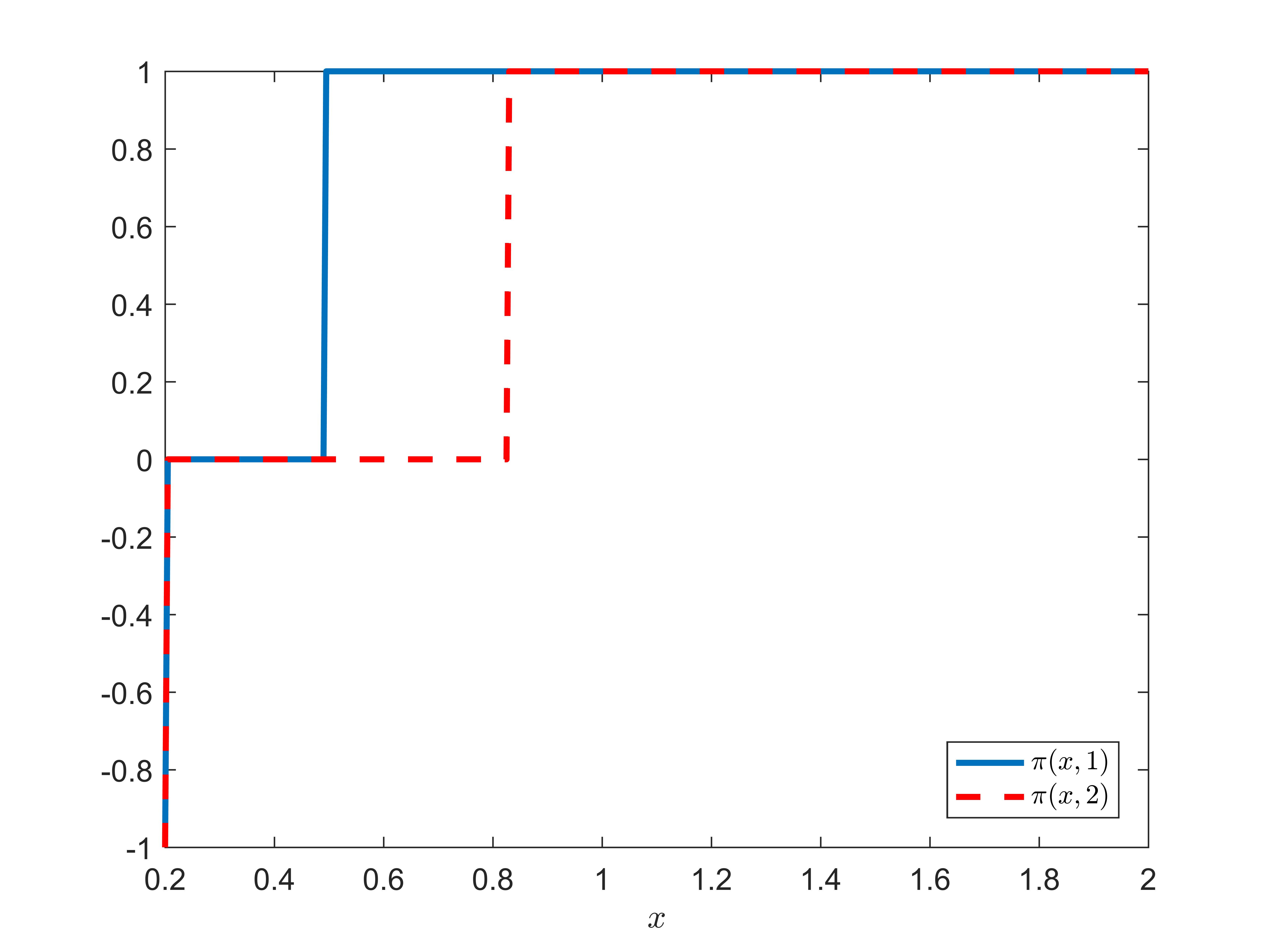} }}%
			\quad
\subfloat{{\includegraphics[width=5cm]{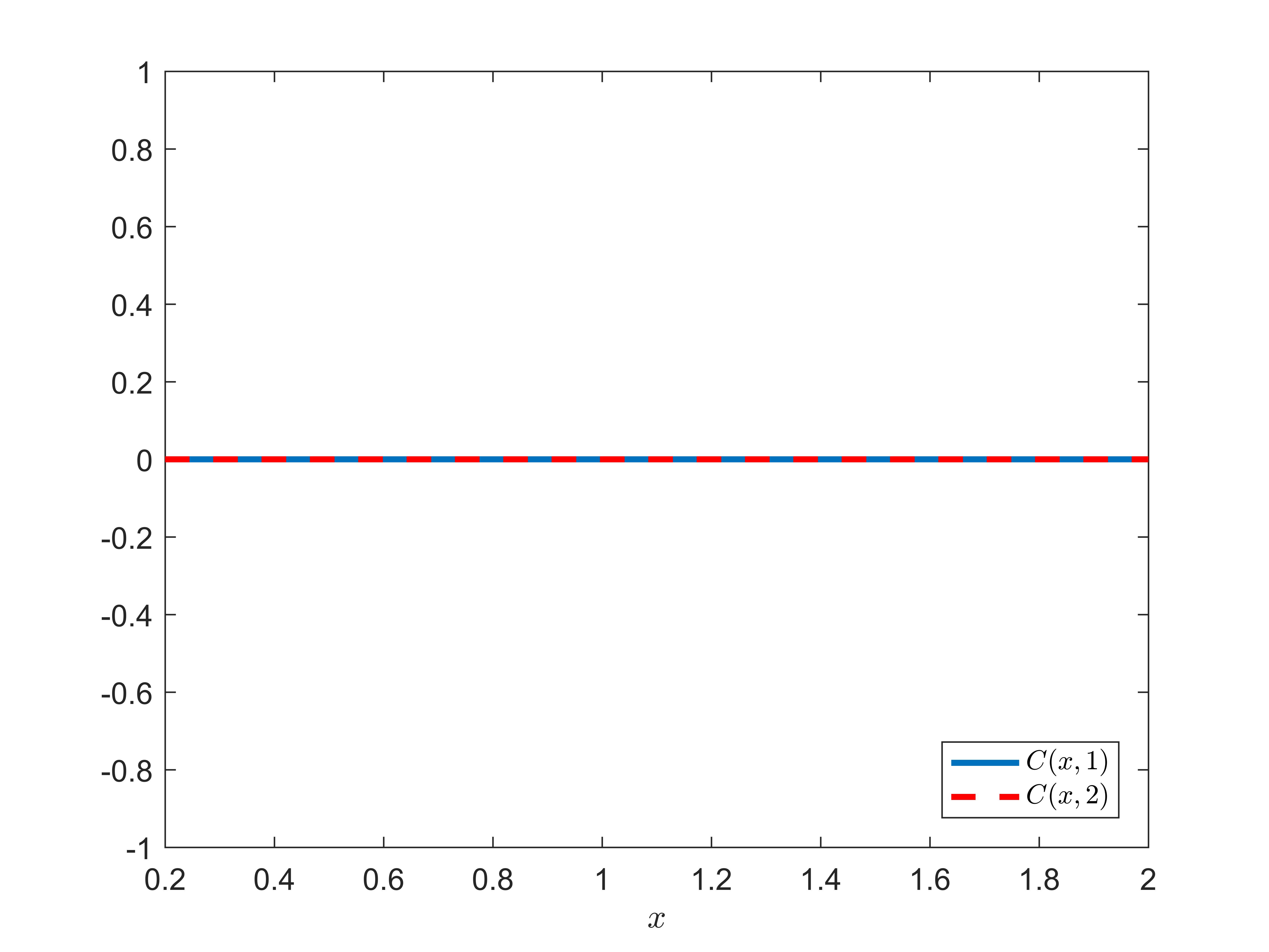} }}%
	\caption{The control type (left) and regular control (right) as functions of  $(x, \al)$ for $\lambda=0.2$ and $\mathcal{U}=\{0\}$  (Example \ref{exm:1})}
	\label{fig1}
\end{figure}

 \begin{figure}[h!tb]
	\begin{center}
		\includegraphics[height=2.5in,width=3.5in]{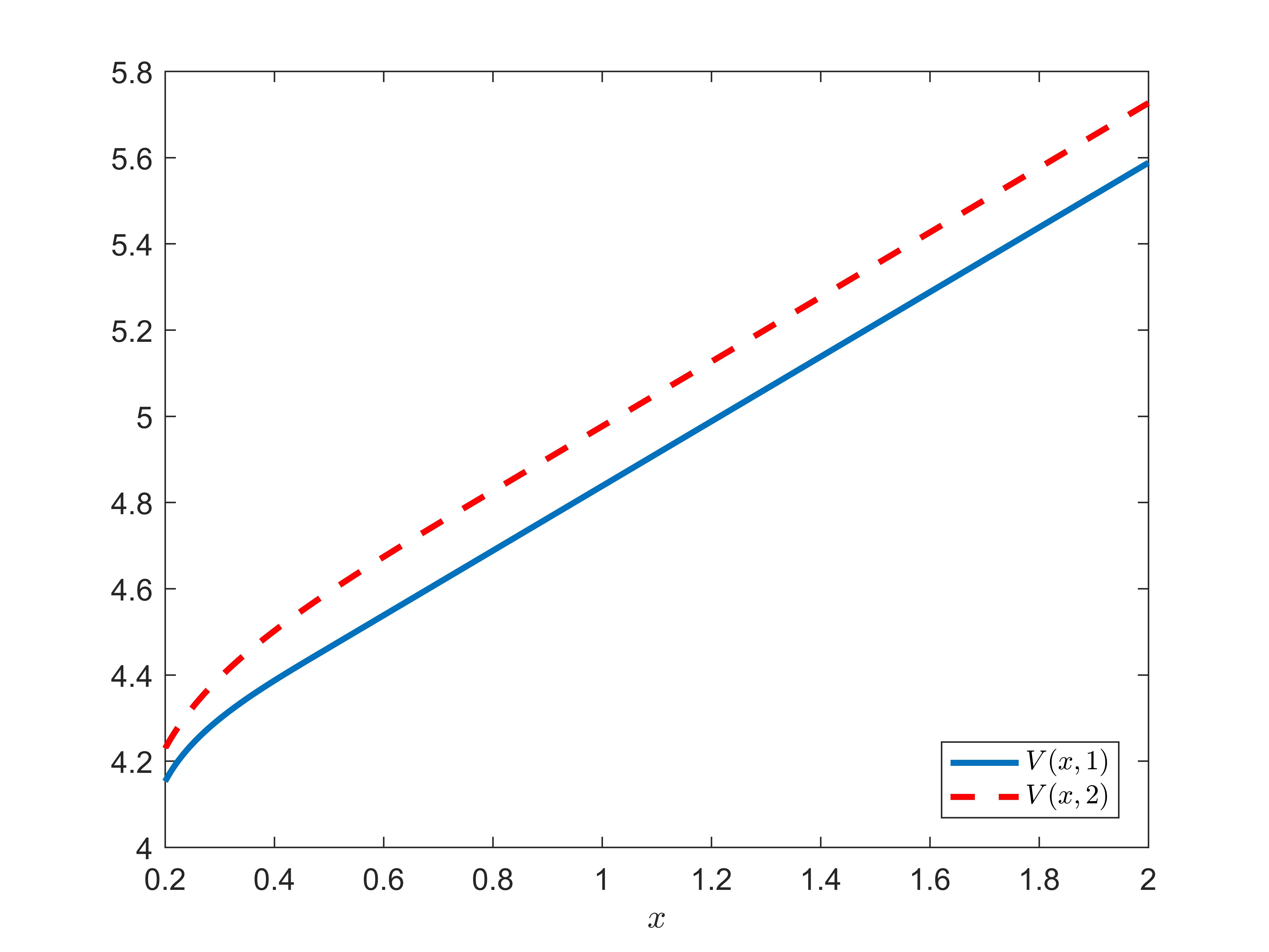}
		\caption{The value function as a function of $(x, \al)$ for $\lambda=0.2$ and $\mathcal{U}=\{0\}$ (Example \ref{exm:1})}  \label{fig2}
		\end{center}
\end{figure}

 \begin{exm} \label{exm:1}
 {\rm Let $\lambda=0.2$ and $\mathcal{U}=\{0\}$. In view of
Theorem \ref{thm:3}, we can take $U=2$.
  Figure \ref{fig1} shows the control type, regular control rate as functions of population size $x$ and  regime $\al$. The corresponding value function is given in Figure \ref{fig2}.
  Since $\mathcal{U}=\{0\}$, the regular control does not work.
   The controller can harvest or renew by exercising impulsive controls only. It appears that in both regimes, if $x=\lambda=0.2$, an impulsive renewing is performed. There is a level  $U_\al$ depending on $\al$ so that we should apply an impulsive harvesting whenever the population size is above $U_\al$ and should not  harvest nor renew whenever the population size is in $(\lambda, U_\al)$. Thus, under this strategy, the population size is in $[\lambda, U_\al]$ for $t\in (0, \infty)$.
  }
 \end{exm}

\begin{figure}[h!tb]
 	\centering		\subfloat{{\includegraphics[width=5cm]{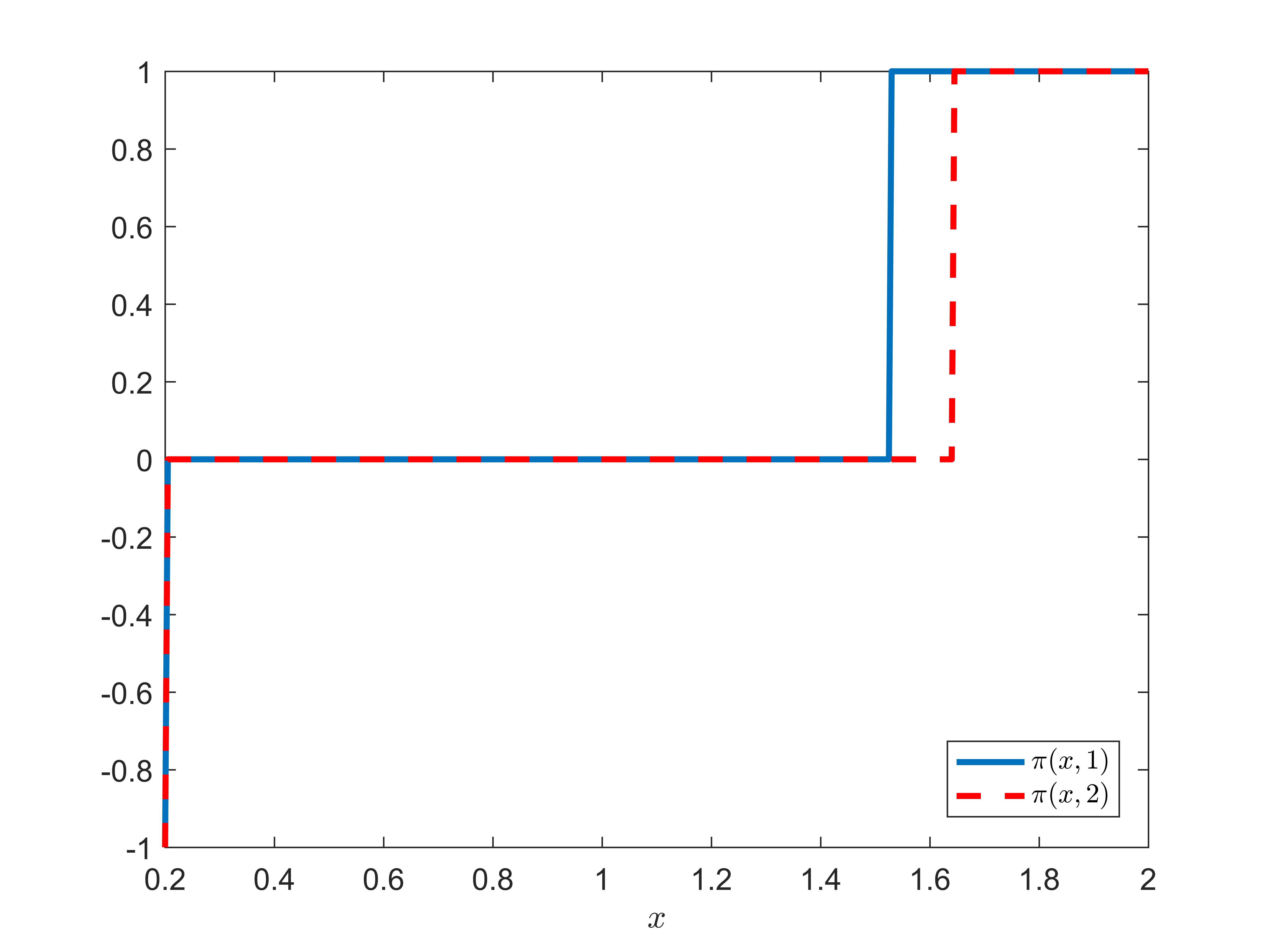} }}%
			\quad
\subfloat{{\includegraphics[width=5cm]{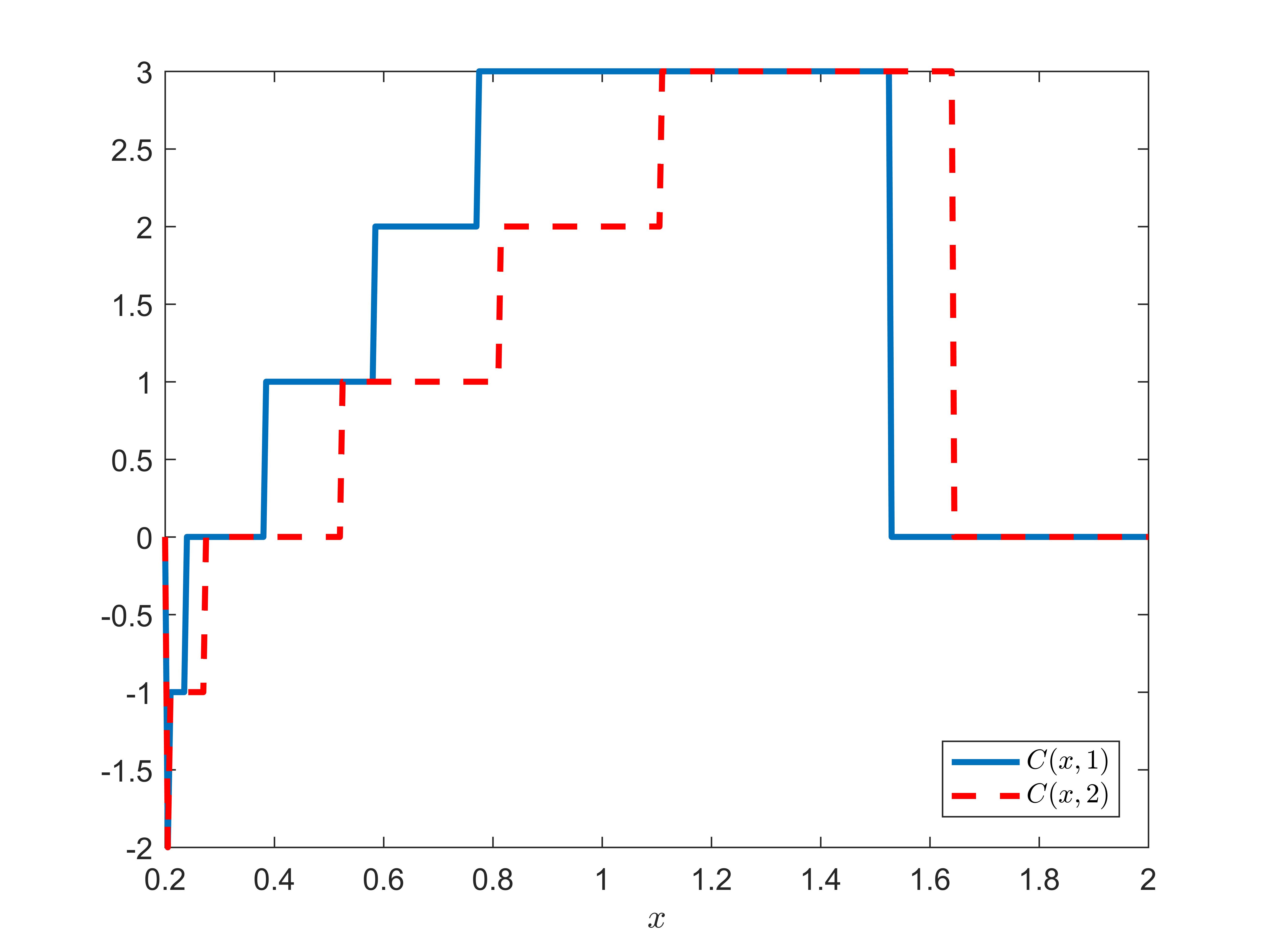} }}%
	\caption{The control type (left) and regular control (right) as functions of $(x, \al)$ for $\lambda =0.2$ and $\mathcal{U}=\{k\in \mathbb{Z}: -2\le k\le 3\}$ (Example \ref{exm:2})}
	\label{fig3}
\end{figure}
	
  \begin{figure}[h!tb]
	\begin{center}
		\includegraphics[height=2.5in,width=3.5in]{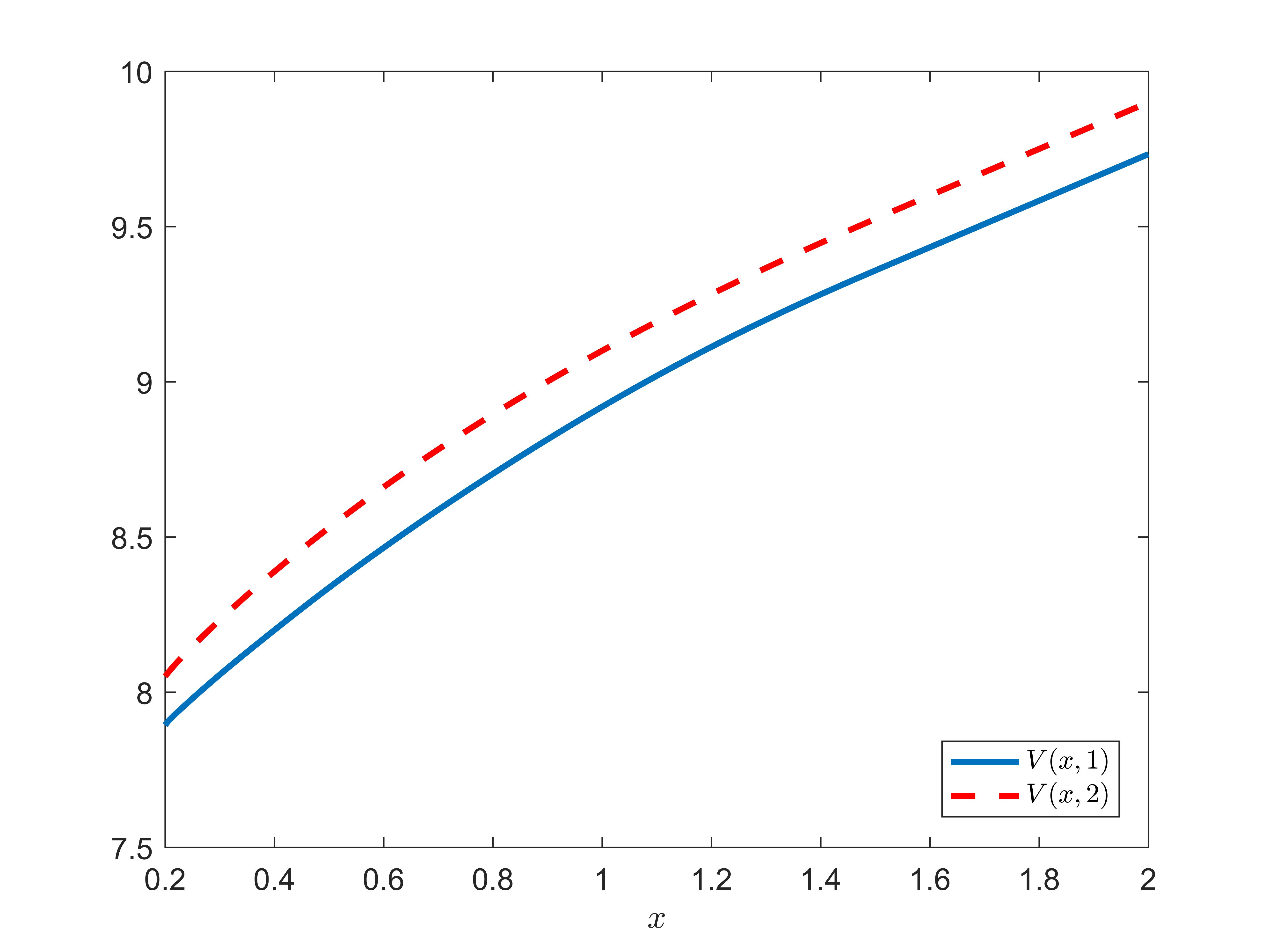}
		\caption{The value function as a function of $(x, \al)$ for $\lambda =0.2$ and $\mathcal{U}=\{k\in \mathbb{Z}: -2\le k\le 3\}$ (Example \ref{exm:2})}  \label{fig4}
		\end{center}
\end{figure}

\begin{exm} \label{exm:2}
{\rm
Let $\lambda =0.2$ and $\mathcal{U}=\{k\in \mathbb{Z}: -2\le k\le 3\}$.  In view of
Theorem \ref{thm:3}, we can take $U=2$.
  Figure \ref{fig3} shows the control type, regular control rate as functions of population size $x$ and regime $\al$. The corresponding value function is given in Figure \ref{fig4}. Compared to the preceding example, the controller have more options to
 harvest or renew. Therefore, the value function in Figure \ref{fig4} is much larger than the preceding one. Moreover, as shown in Figure \ref{fig3}, the larger the population is, the higher the harvesting rate is.
  }
  \end{exm}

  \begin{figure}[h!tb]
 	\centering		\subfloat{{\includegraphics[width=5cm]{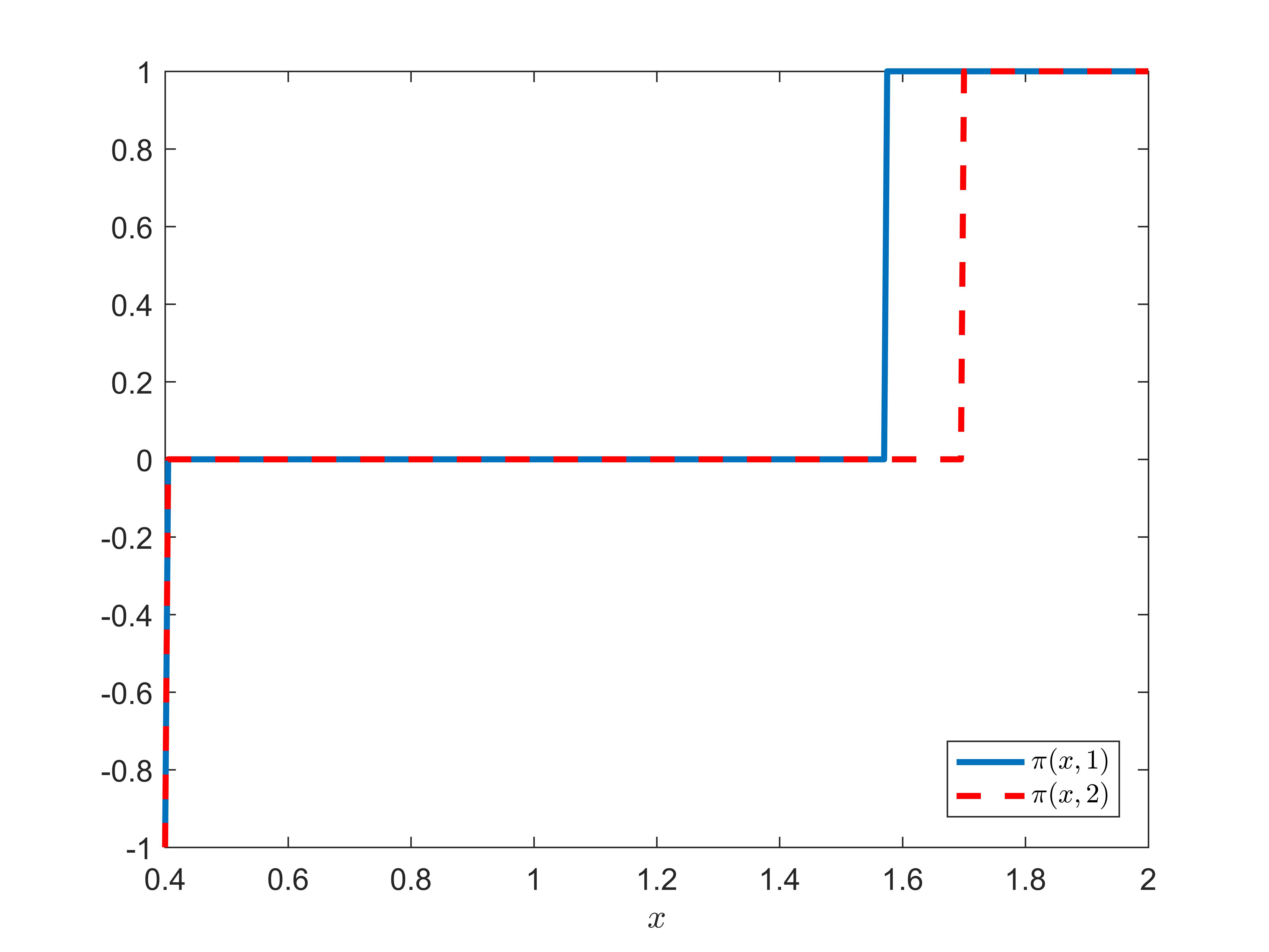} }}%
			\quad
\subfloat{{\includegraphics[width=5cm]{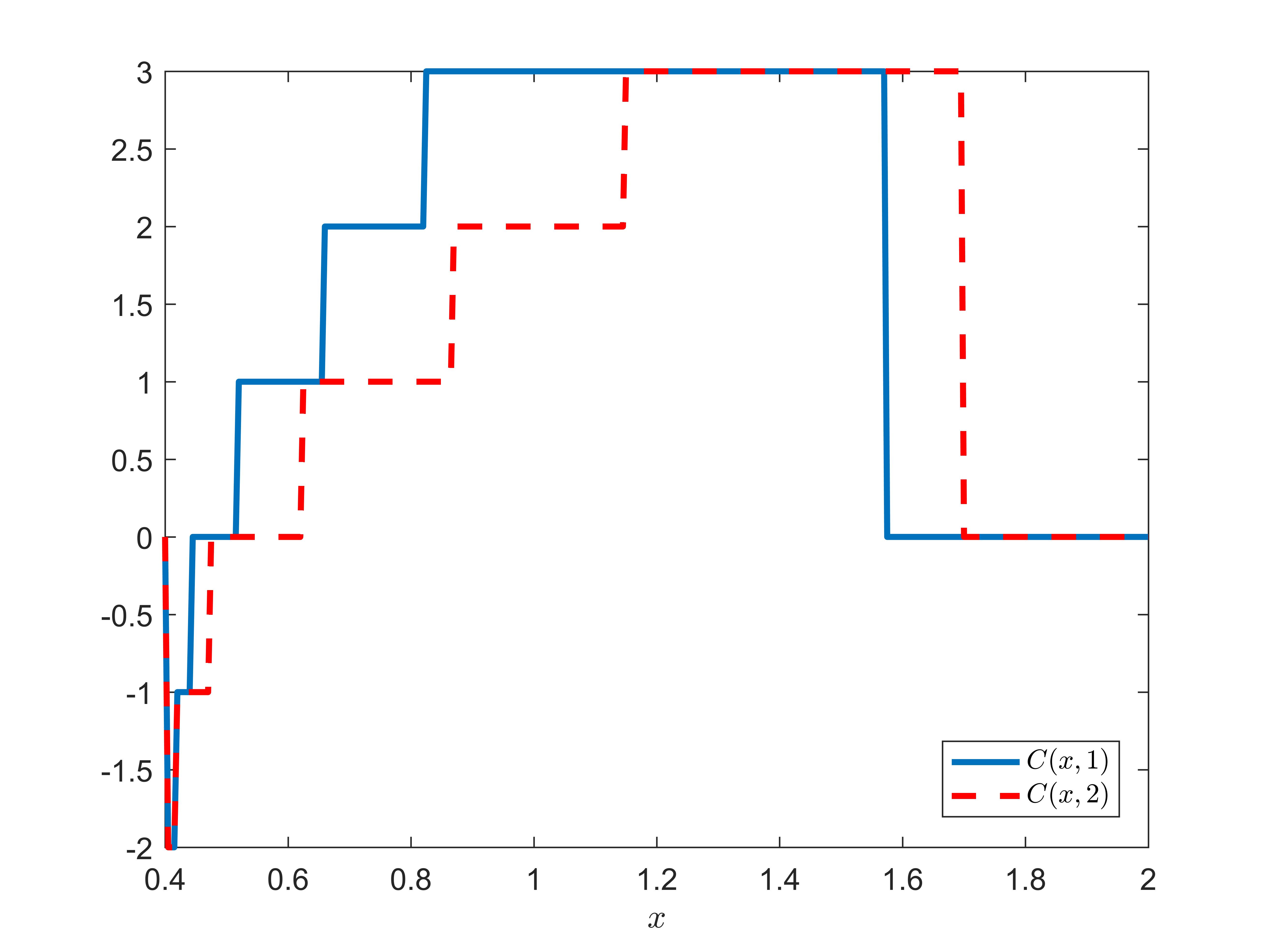} }}%
	\caption{The control type (left) and regular control (right) as functions of $(x, \al)$ for $\lambda =0.2$ and $\mathcal{U}=\{k\in \mathbb{Z}: -2\le k\le 3\}$ (Example \ref{exm:3})}
	\label{fig5}
\end{figure}
	
  \begin{figure}[h!tb]
	\begin{center}
		\includegraphics[height=2.5in,width=3.5in]{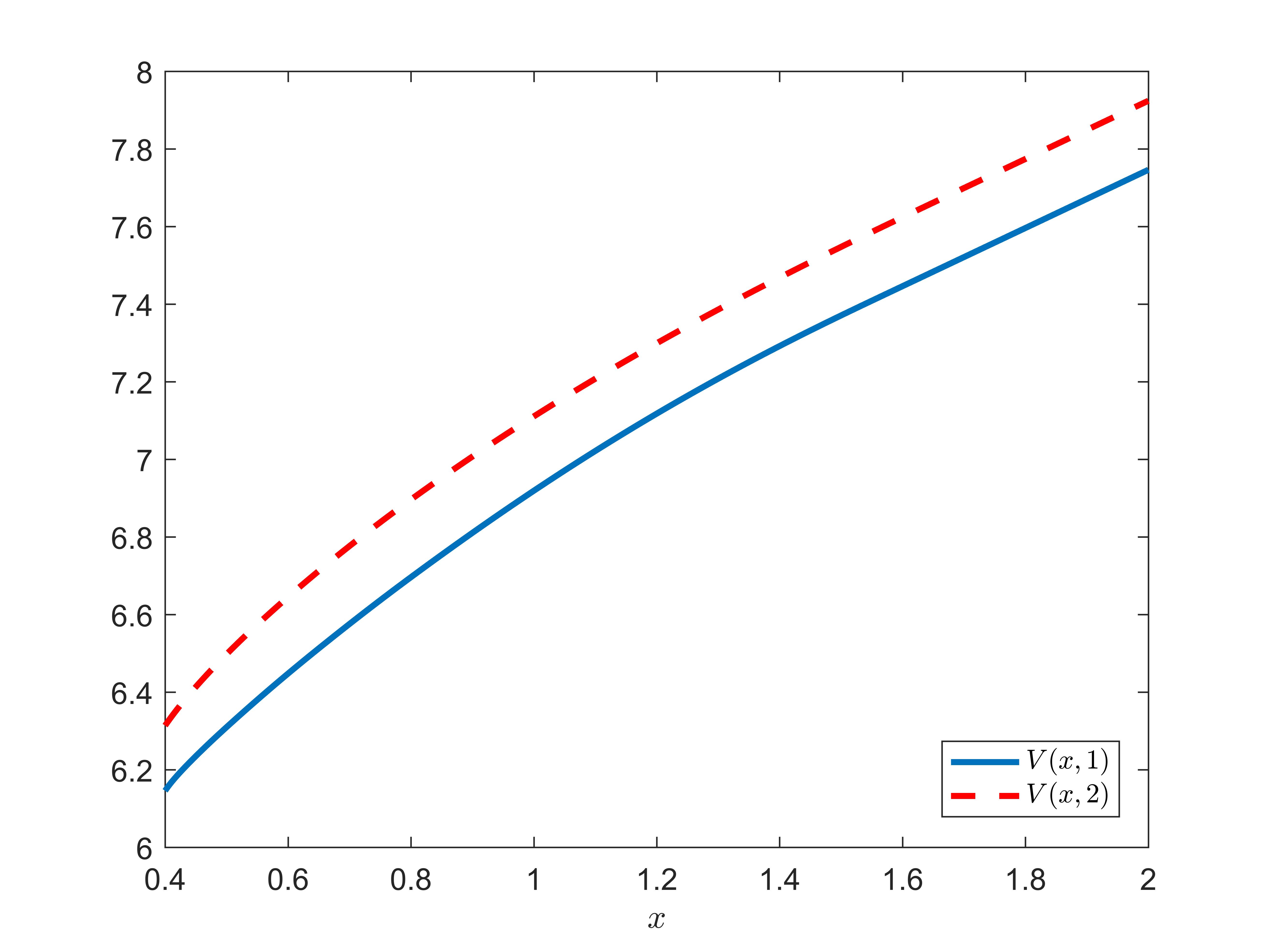}
		\caption{The value function as a function of  $(x, \al)$ for $\lambda =0.2$ and $\mathcal{U}=\{k\in \mathbb{Z}: -2\le k\le 3\}$ (Example \ref{exm:3})}  \label{fig6}
		\end{center}
\end{figure}

\begin{exm} \label{exm:3}
{\rm
Let $\lambda =0.4$ and $\mathcal{U}=\{k\in \mathbb{Z}: -2\le k\le 3\}$. In view of
Theorem \ref{thm:3}, we can take $U=2$.
  Figure \ref{fig5} shows the control type, regular control rate as functions of population size $x$ and regime $\al$. The corresponding value function is given in Figure \ref{fig6}.
  It can be seen that the value function is smaller than that in Example \ref{exm:2}. This observation fits the fact that the one can harvest only if the population size is higher than $\lambda =0.4$ compared to 0.2 in Example \ref{exm:2}.  In this and the previous numerical experiments, we observe the same phenomena as follows: (a) if the population size hits $\lambda$, an impulsive renewing is performed to keep the species in $\SS=[\lambda, \infty)$, which is the harvesting-renewing domain; (b) For each $\al$, there are levels $L^{(1)}_\al, L^{(2)}_\al$,  and $L^{(3)}_\al$ so that we should renew with bounded rates if $x\in(\lambda, L^{(1)}_\al)$, we should not harvest nor renew if $x\in[L^{(1)}_\al, L^{(2)}_\al)$,   we should harvest with bounded rates if $x\in[L^{(2)}_\al, L^{(3)}_\al)$, and we should perform an impulsive harvesting if $x\in[L^{(3)}_\al, U]$.

  It can be seen from the numerical experiments that we should keep the population size in the interval $[\lambda, U]$ for $t>0$. In other words, if the initial population size  $x>\lambda$, then $\tau=\infty$ almost surely. Thus, compared with the well-known formulations for harvesting-type problems,
 the proposed model offers an effective planning for the balance between economical aspect and the
 sustainable purpose.
  }
  \end{exm}

 \begin{exm} \label{exm:4}
{\rm
 We
consider Eq. \eqref{e.exm} with
 \begin{equation*}\barray
& b(x, \al) = x(\al-1.5x), \ \  \sg(x, \al) = \mu_1 x + \mu_2 \ \ \text{for} \ \ (x, \al)\in \R_+\times \M,
\earray
\end{equation*}
where the constants $\mu_1$ and $\mu_2$ are 
to be determined. 
The regular cost function is $g(x, \al, c) = \frac{c^2}{10}$, the control set is $\mathcal U = \{k/2: -4\le k\le 4, k\in \mathbb{Z}\}$, and we keep the other data as in the preceding examples. To explore how noise impacts the problem, we first choose $(\mu_1, \mu_2)=(1,0)$, and record the results in   Figure \ref{fig7}. The results for large white noise intensities
when $(\mu_1, \mu_2)=(30, 0)$ and $(\mu_1, \mu_2)=(0, 30)$ are presented in Figures \ref{fig8} and  \ref{fig9}, respectively. It turns out, as stated 
in 
Theorems \ref{thm:4} and \ref{thm:5}, when the white noise intensity is very large, the value function $V (x, \al)$ is close to $q(x-\lambda)$ for $(x, \al)\in [\lambda, \infty)\times \M$ (see Figures \ref{fig8} and  \ref{fig9}). Also, as shown in Figures \ref{fig8} and  \ref{fig9}, under a large white noise intensity, one should always harvest (either with the maximal bounded harvesting rate  $C(x, \al) =2$ or with an impulsive harvesting) and never renew.  Meanwhile, as shown in  Figure \ref{fig7}, when the noise intensity is not large, the regular control takes a variety of values in the control set $\mathcal{U}$ and the value function is much higher than those in  Figures \ref{fig8} and  \ref{fig9}.
We refer to \cite{A1998,H2019,H2020} for more insight and numerical experiments
regarding how noise can impact harvesting and renewing actions.

  \begin{figure}[h!tb]
	\begin{center}
		\includegraphics[height=2.5in,width=5in]{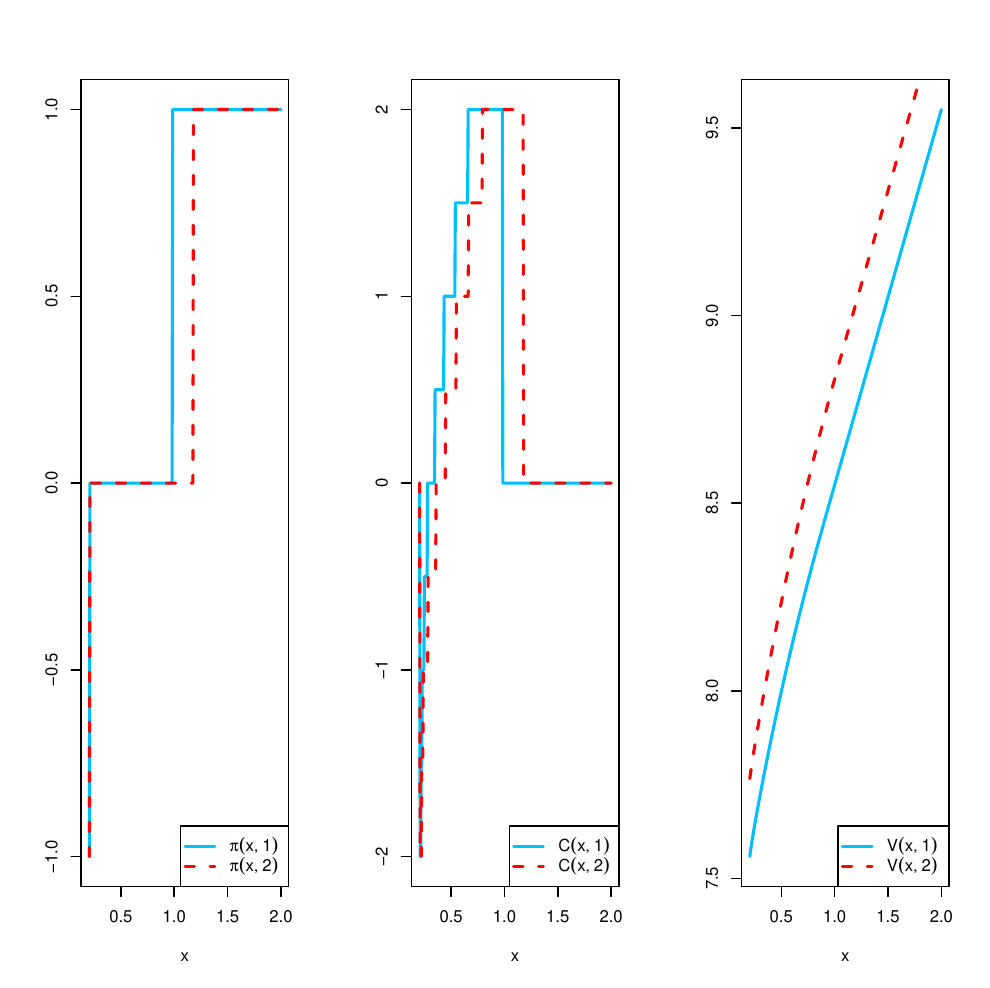}
		\caption{The control type (left), regular control (middle), and  the value function (right) as functions of  $(x, \al)$ for $\lambda =0.2$, $\mathcal{U}=\{k/2: -4\le k\le 4, k\in \mathbb{Z}\}$, $(\mu_1, \mu_2)=(1, 0)$ (Example \ref{exm:4})}  \label{fig7}
		\end{center}
\end{figure}

 \begin{figure}[h!tb]
	\begin{center}
		\includegraphics[height=2.5in,width=5in]{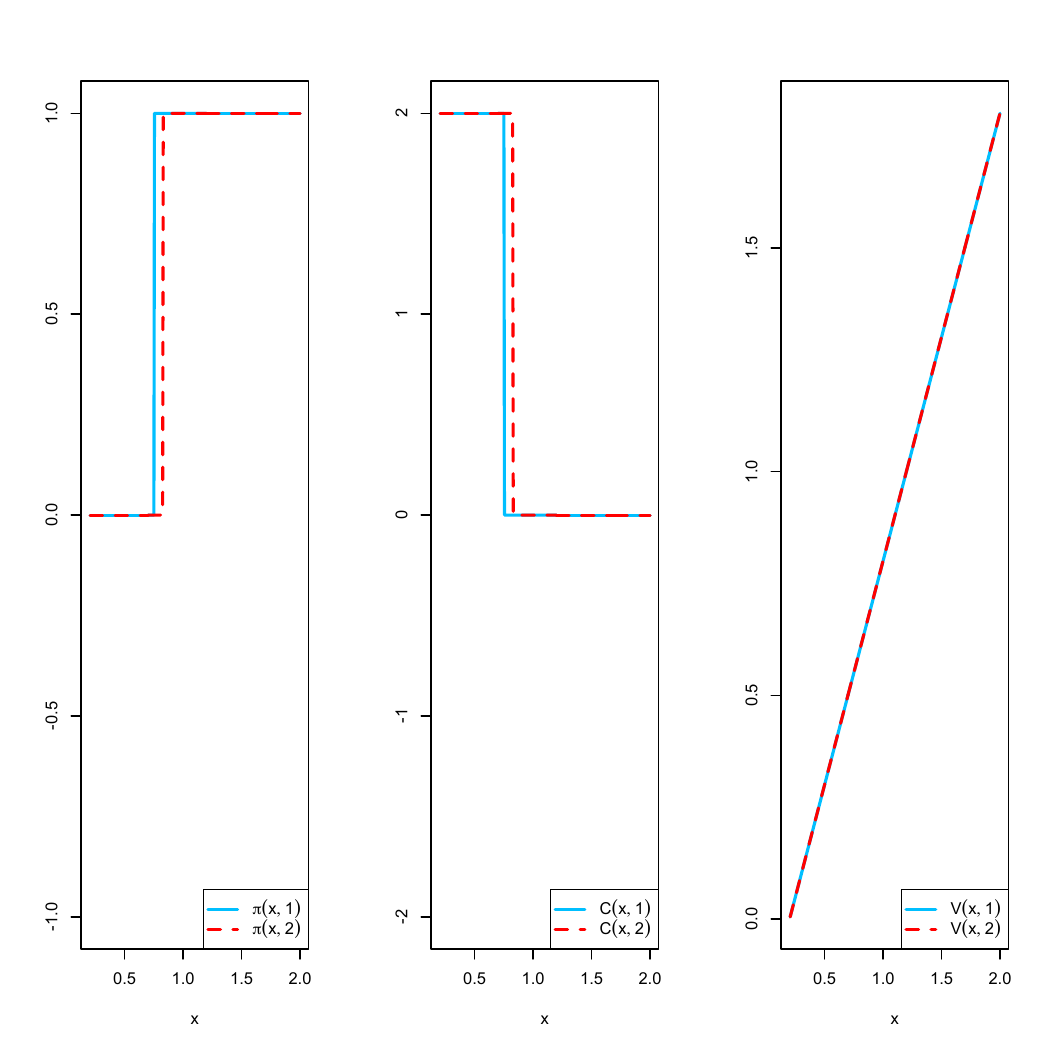}
		\caption{The control type (left), regular control (middle), and  the value function (right) as functions of  $(x, \al)$ for $\lambda =0.2$, $\mathcal{U}=\{k/2: -4\le k\le 4, k\in \mathbb{Z}\}$, $(\mu_1, \mu_2)=(30, 0)$ (Example \ref{exm:4})}  \label{fig8}
		\end{center}
\end{figure}

 \begin{figure}[h!tb]
	\begin{center}
		\includegraphics[height=2.5in,width=5in]{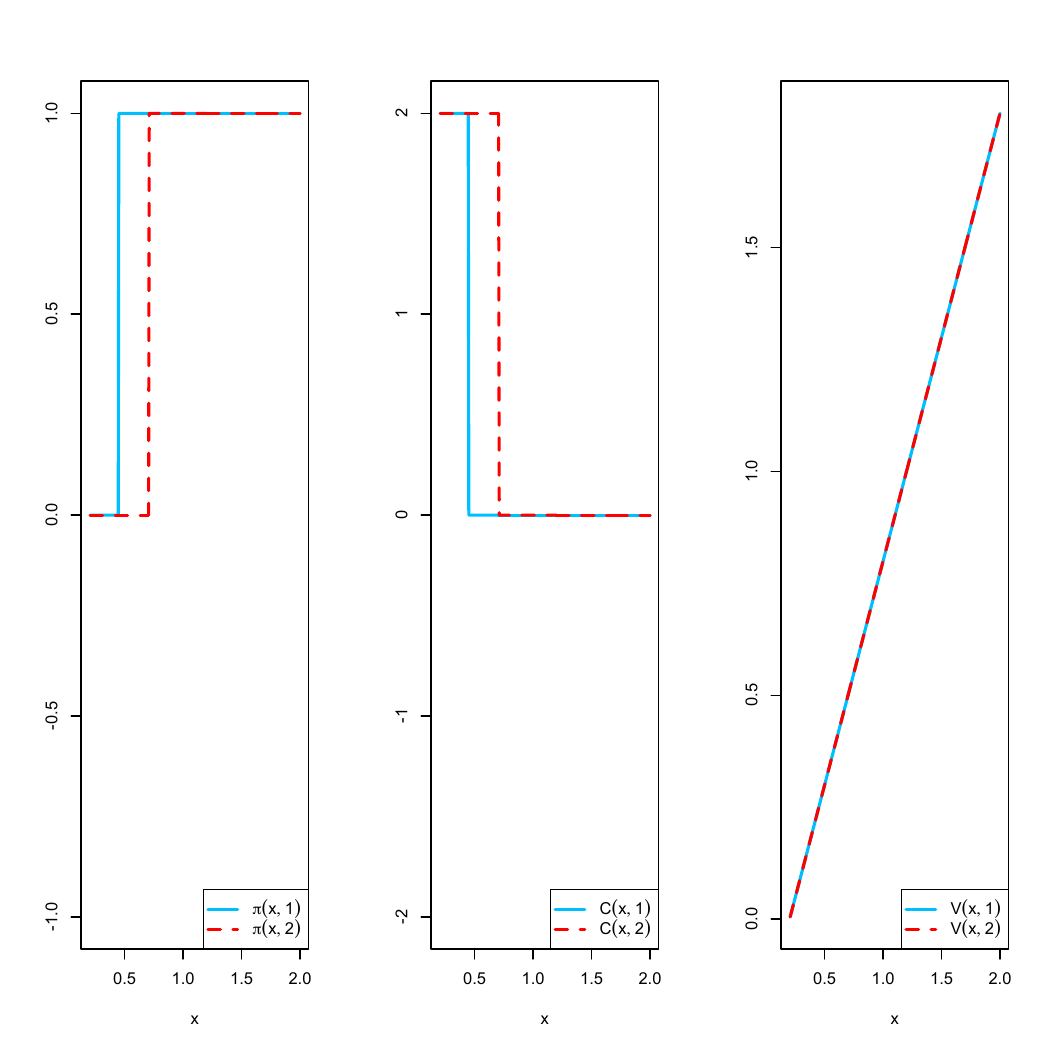}
		\caption{The control type (left), regular control (middle), and  the value function (right) as functions of  $(x, \al)$ for $\lambda =0.2$, $\mathcal{U}=\{k/2: -4\le k\le 4, k\in \mathbb{Z}\}$, $(\mu_1, \mu_2)=(0, 30)$ (Example \ref{exm:4})}  \label{fig9}
		\end{center}
\end{figure}

  }
   \end{exm}

\section{Further Remarks}

This paper has been devoted to the study of a generalized harvesting problem for  a stochastic population. We have established the finiteness and continuity of the value function. Moreover, we have shown that for common population systems, it is important to maintain the population size  in a bounded set. We have also studied the impact of large white noise on harvesting.
A numerical algorithm has been constructed by the Markov chain approximation methods. The numerical experiments have revealed the effect of Markovian switching and control costs associated with harvesting/renewing activities. It has been observed that under a mixed singular-control formulation, the decision makers have more options to harvest or renew the species, which is much more beneficial than the known models with no renewing, or with only one control component. Moreover, the state constraint provides an effective strategy for sustainable purpose.

In this paper, we focus on single species. The development in this paper can be carried over to the multidimensional case. 
As for single species models (Theorem \ref{thm:3}), it appears that the controller should keep the population sizes in a bounded set. However,
in order to identify such a bounded set,
one needs to be careful to handle the interaction between species.
Although multidimensional systems can be handled,
the
multidimensional structure added more difficulty to study the impact of large white noise. The computation of multidimensional ecosystems with many constraints and controls is another challenging task.
In addition, one can also consider the problem under other constraints, random prices, and random cost functions.



\begin{thebibliography}{plain}










\bibitem{A1998}  Alvarez, L.H.R., Shepp, L.A.: Optimal harvesting of stochastically fluctuating populations. J. Math. Biol. 37, 155-177 (1998)


\bibitem{A2000} Alvarez, L.H.R.: Singular stochastic control in the presence of a state-dependent yield structure. Stochastic Process. Appl. 86, 323-343 (2000)


\bibitem{Lande95}
Lande, R.,  Engen, S.,  S\ae ther, B.: Optimal harvesting of fluctuating populations with a risk of extinction. Am. Naturalist. 145(5), 728-745 (1995)



\bibitem{F2019}
 Ferrari, G.: On a class of singular stochastic control problems for reflected diffusions. J. Math. Anal. Appl.  473(2), 952-979 (2019)




\bibitem{Fleming93}  Fleming, W.H.,  Soner, H.M.: Controlled Markov Processes and Viscosity Solutions,
Springer, New York (1993)




\bibitem{H2019b} Alvarez, L.H.R., Hening, A.:Optimal sustainable harvesting of populations in random environments. Stochastic Process. Appl. 150, 678-698 (2022)



\bibitem{H2019} Hening, A., Tran, K., Phan, T.T., Yin, G.: Harvesting of interacting stochastic populations.  J. Math. Biol.  79(2), 533-570 (2019)



\bibitem{H2020}  Hening, A., Tran, K.: Harvesting and seeding of stochastic populations: analysis and numerical approximation. J. Math. Biol. 81(1), 65-112  (2020)


\bibitem{Hau1995}  Haussmann, U.,  Suo, W.:
Singular optimal stochastic controls. I. Existence. SIAM J. Control Optim. 33(3), 916-936 (1995)


\bibitem{Hau19952}  Haussmann, U.,  Suo, W.: Singular optimal stochastic controls. II. Dynamic programming.  SIAM J. Control Optim. 33(3), 937-959 (1995)


\bibitem{Jin12} Jin, Z.,  Yin, G.,  Zhu, C.: Numerical solutions of optimal risk control and dividend optimization policies under a generalized singular control formulation. Automatica J. IFAC 48(8), 1489-1501 (2012)



 \bibitem{K1984} Karatzas, I., Shreve, S.E.: Connections between optimal stopping and singular stochastic control. I. Monotone follower problems. SIAM J. Control Optim. 22(6), 856-877 (1984)



 \bibitem{K1985} Karatzas, I., Shreve, S.E.: Connections between optimal stopping and singular stochastic control. II. Reflected follower problems. SIAM J. Control Optim. 23(3), 433-451 (1985)




\bibitem{K2020} Kharroubi, I., Lim, T.,  Mastrolia, T.: Regulation of renewable resource exploitation. SIAM J. Control Optim.  58(1) 551-579 (2020)



\bibitem{Kunwai21} Kunwai K.,  Xi, F.,  Yin, G.,  Zhu, C.: On an Ergodic Two-Sided Singular Control Problem.  Appl. Math. Optim. 86, 26 (2022) https://doi.org/10.1007/s00245-022-09881-0



\bibitem{KM91}  Kushner, H.J.,  Martins, L.F.: Numerical methods for stochastic singular control problems. SIAM J. Control Optim.  29(6),  1443-1475 (1991)


\bibitem{Kushner92} Kushner, H.J.,  Dupuis, P.G.: Numerical methods for stochastic control problems in continuous time,
Springer, New York  (1992)




\bibitem{L2020} Liang, G., Zervos, M.: Ergodic singular stochastic control motivated by the optimal sustainable exploitation of an ecosystem. arXiv preprint arXiv:2008.05576 (2020)



\bibitem{L1997}  Lungu, E.M.,  {\O}ksendal, B.: Optimal harvesting from a population in a stochastic crowded environment. Math. Biosci. 145(1), 47-75 (1997)


\bibitem{Mao2006} Mao, X., Yuan, C.: Stochastic differential equations with Markovian switching, Imperial College Press, London (2006)



\bibitem{N2020}  Nguyen, D.H., Yin, G.: Sustainable harvesting policies under long-run average criteria: near optimality. Appl. Math. Optim. 81(2), 443-478 (2020)





    \bibitem{NNY21}
 Nguyen, D., Nguyen, N., Yin, G.:
Stochastic functional Kolmogorov equations I: Persistence. Stochastic Process. Appl. 142, 319-364 (2021)



\bibitem{Pham09}
Pham, H.,  {\sl Continuous-Time Stochastic Control and Optimization with Financial Applications}, Stoch. Model. Appl. Probab. 61, Springer-Verlag, Berlin, 2009.




\bibitem{RS1996} Radner, R.,  Shepp, L.: Risk vs. profit potential: A model for corporate strategy. J. Econom. Dynam. Control 20(8), 1373-1393 (1996)





\bibitem{Zhu11} Song, Q.,  Stockbridge, R.H., Zhu, C.: On optimal harvesting problems in random environments.  SIAM J. Control Optim. 49(2), 859-889 (2011)




\bibitem{Ky17} Tran, K., Yin, G.: Optimal harvesting strategies for stochastic ecosystems. IET Control Theory Appl. 11(15), 2521-2530 (2017)


\bibitem{Ky21} Tran, K.: Optimal exploitation for hybrid systems of renewable resources under partial observation.  Nonlinear Anal. Hybrid Syst. 40 101013 (2021)

\bibitem{YinZ10} Yin, G., Zhu, C.: Hybrid Switching Diffusions: Properties and
Applications, Springer, New York  (2010)


\end{thebibliography}
\end{document}